\documentclass[a4paper]{amsart}
\usepackage{wrapfig}
\usepackage[dvips]{graphicx}
\usepackage{amsmath,amsthm,amssymb,amscd}
\usepackage{mathrsfs}
\theoremstyle{definition}
\newtheorem{thm}{Theorem}[section]
\newtheorem{Def}[thm]{Definition}
\newtheorem{pro}[thm]{Proposition}
\newtheorem{cor}[thm]{Corollary}
\newtheorem{lem}[thm]{Lemma}

\newtheorem{rem}[thm]{Remark}
\newtheorem{que}[thm]{Question}
\newtheorem*{mainthm}{Theorem}
\theoremstyle{definition}

\begin{document}
\title{Trace scaling automorphisms of the stabilized Razak-Jacelon algebra}
\author{Norio Nawata}
\address{Department of Educational Collaboration, Osaka Kyoiku University, 4-698-1 Asahigaoka, Kashiwara, Osaka, 582-8582, Japan}
\email{nawata@cc.osaka-kyoiku.ac.jp}
\thanks{This work was supported by JSPS KAKENHI Grant Number 16K17614}

\begin{abstract}
We classify trace scaling automorphisms of $\mathcal{W}\otimes\mathbb{K}$ up to outer conjugacy, where 
$\mathcal{W}$ is a certain simple separable nuclear stably projectionless C$^*$-algebra having trivial $K$-groups. 
Also, we show that all automorphisms of $\mathcal{W}$ with the Rohlin property are outer conjugate to each other. 
Moreover, we show that the central sequence C$^*$-algebra $F(\mathcal{W})$ of $\mathcal{W}$ is infinite, which answers 
a question of Kirchberg. 
\end{abstract}
\maketitle

\section{Introduction} 

Let $\mathcal{W}$ be the Razak-Jacelon algebra studied in \cite{J}, which is a certain simple separable nuclear stably projectionless C$^*$-algebra having trivial $K$-groups
and a unique tracial state and no unbounded traces. We may regard $\mathcal{W}$ as a stably finite analogue of the Cuntz algebra $\mathcal{O}_2$. 
Note that a C$^*$-algebra $A$ is said to be \textit{stably projectionless} if $A\otimes\mathbb{K}$ has no non-zero projections, 
where $\mathbb{K}$ is the C$^*$-algebra of compact operators on an infinite-dimensional separable Hilbert space. In particular, 
every stably projectionless C$^*$-algebra is non-unital. We refer the reader to \cite{EN}, \cite{GL} and \cite{GL2} for remarkable progress in the 
classification of such C$^*$-algebras. 

In this paper, we study trace scaling automorphisms of $\mathcal{W}\otimes\mathbb{K}$ and show that these automorphisms are outer conjugate if and only if the scaling factors 
coincide. This classification can be regarded as an analogous result of Connes' classification \cite{C2} of trace scaling automorphisms of the AFD factor of type II$_{\infty}$. 
In the case of C$^*$-algebras, Elliott, Evans and Kishimoto \cite{EEK} classified trace scaling automorphisms of stable UHF algebras. Moreover, Evans and Kishimoto \cite{EK} classified 
trace scaling automorphisms of stable AF algebras with totally ordered $K_0$-groups. (See also \cite{BK}.) 
More generally, the study of group actions on operator algebras is one of the most fundamental subjects and has a long history in the theory of operator algebras. 
We refer the reader to \cite{I} and the references given there for this subject. 
We recall some other classification results of automorphisms of C$^*$-algebras. 
Kishimoto \cite{Kis0} showed that if $\alpha$ and $\beta$ are automorphisms of a UHF algebra such that $\alpha^m$ and $\beta^{m}$ are 
strongly outer for any $m\in\mathbb{Z}\setminus\{0\}$, then $\alpha$ and $\beta$ are outer conjugate. 
Moreover, Kishimoto  classified a large class of automorphisms of certain A$\mathbb{T}$ algebras in \cite{Kis1} and \cite{Kis2}. 
Matui \cite{M1} generalized this result to certain simple AH algebras. Nakamura \cite{Nak} completely classified aperiodic automorphisms of Kirchberg algebras. 
Sato \cite{Sa} showed that if $\alpha$ and $\beta$ are automorphisms of the Jiang-Su algebra $\mathcal{Z}$ such that $\alpha^m$ and $\beta^{m}$ are 
strongly outer for any $m\in\mathbb{Z}\setminus\{0\}$, then $\alpha$ and $\beta$ are outer conjugate. 
Note that it is important to consider the Rohlin property or variants thereof for classifying automorphisms of operator algebras. 

If $A$ is stably projectionless, then the central sequence C$^*$-algebra $A_{\omega}$ of $A$ is also stably projectionless. 
Hence $A_{\omega}$ is not very useful for our purpose. 
In this paper, Kirchberg's central sequence C$^*$-algebra \cite{Kir2} plays a central role. 
Kirchberg's central sequence C$^*$-algebra $F(A)$ is defined as the quotient C$^*$-algebra of $A_{\omega}$ by the annihilator of $A$. 

This paper is organized as follows:  In Section \ref{sec:Pre}, we collect notations and some results.  
In Section \ref{sec:stable}, we review some results in \cite{EN} and reformulate for our purpose. Note that our arguments are essentially based on 
Elliott and Niu's arguments. 
In Section \ref{sec:Properties}, we study properties of $F(\mathcal{W})$. 
We show that $F(\mathcal{W})$ has many projections (Proposition \ref{pro:key-pro}). This is an easy corollary of Razak's classification theorem \cite{Raz} and 
Matui and Sato's result in \cite{MS2}. But this property enables us to deal with $F(\mathcal{W})$ like a C$^*$-algebra of real rank zero. 
In Section \ref{sec:Homotopy}, we obtain a homotopy type theorem for unitaries in $F(\mathcal{W})$ by classifying certain unitaries in $F(\mathcal{W})$ up to unitary equivalence 
(Theorem \ref{thm:homotopy}, Theorem \ref{thm:unitary-equivalence}). Also, we classify certain projections in $F(\mathcal{W})$ up to unitary equivalence and show that 
the unit $1$ in $F(\mathcal{W})$ is infinite (Theorem \ref{thm:unitary-equivalence-projections}, Corollary \ref{cor:Kirchberg-question}). This is an answer to \cite[Question 2.14]{Kir2}. 
Some arguments in this section are motivated by arguments in \cite[Section 4]{M2} (see also \cite{L}). 
In Section \ref{sec:Rohlin}, we introduce the Rohlin property for automorphisms of separable C$^*$-algebras and show that 
every trace scaling automorphism of $\mathcal{W}\otimes\mathbb{K}$ has the Rohlin property (Theorem \ref{thm:rohlin-type-scale}). 
Moreover, we show that if $\alpha$ is an automorphism of $\mathcal{W}$ such that $\alpha^m$ is strongly outer for any $m\in\mathbb{Z}\setminus\{0\}$, 
then $\alpha$ has the Rohlin property (Theorem \ref{thm:rohlin-type}). 
In Section \ref{sec:Outer}, we obtain a classification theorem (Theorem \ref{thm:main-i}) of trace scaling automorphisms of $\mathcal{W}\otimes\mathbb{K}$ by using the 
Bratteli-Elliott-Evans-Kishimoto intertwining argument. 
By the uniqueness of traces on $\mathcal{W}\otimes\mathbb{K}$, for any automorphism $\alpha$ of $\mathcal{W}\otimes\mathbb{K}$, 
there exists a positive real number $\lambda (\alpha)$ such that $\tau\otimes\mathrm{Tr} \circ \alpha =\lambda(\alpha) \tau\otimes\mathrm{Tr}$. 
We say that $\alpha$ is a \textit{trace scaling automorphism} if $\lambda (\alpha)\neq 1$. The following theorem is the main result in this paper. 

\begin{mainthm} \ \\
Let $\alpha$ and $\beta$ be trace scaling automorphisms of $\mathcal{W}\otimes\mathbb{K}$. Then $\alpha$ and $\beta$ are outer conjugate if and only if 
$\lambda (\alpha)=\lambda (\beta)$. 
\end{mainthm} 

The range of the invariant $\{\lambda(\alpha)\in\mathbb{R}_{+}^{\times}\; |\; \alpha\in \mathrm{Aut}(\mathcal{W}\otimes\mathbb{K})\}$ is equal to the fundamental group $\mathcal{F}(\mathcal{W})$ of 
$\mathcal{W}$, which is introduced in \cite{Na1} (see also \cite[Proposition 2.8]{Na2}). 
By Razak's classification theorem \cite{Raz} and Robert's classification theorem \cite{Rob}, $\mathcal{F}(\mathcal{W})$ is equal to $\mathbb{R}^\times_{+}$. 
Moreover, combining the results of Kishimoto-Kumjian \cite{KK1}, \cite{KK2}, Dean \cite{Dean} and Robert \cite{Rob}, we see that there exists a trace scaling flow on $\mathcal{W}\otimes\mathbb{K}$. 
Note that a separable C$^*$-algebra with the uncountable fundamental group must be stably projectionless (see \cite[Corollary 4.10]{Na1}). 
Hence stably projectionless C$^*$-algebras seem to be more analogous to the AFD factors of type II than (finite) stably unital C$^*$-algebras. 
We also show that if $\alpha$ and $\beta$ are automorphisms of $\mathcal{W}$ such that $\alpha^m$ and $\beta^{m}$ are 
strongly outer for any $m\in\mathbb{Z}\setminus\{0\}$, then $\alpha$ and $\beta$ are outer conjugate (Theorem \ref{thm:main-ii}). 

After the first version of this paper was on the arxiv, G\'abor Szab\'o generalized some results in this paper to all classifiable $KK$-contractible C$^*$-algebras 
(see \cite[Theorem 5.11 and Theorem 5.12]{Sza}) by using Gong and Lin's basic homotopy lemma \cite{GL}.

\section{Preliminaries}\label{sec:Pre}

In this section we shall collect notations and some results. We refer the reader to \cite{Bla} and \cite{Ped2} for basic facts of operator algebras. 

\subsection{Notation} 
We say that a C$^*$-algebra $A$ is $\sigma$-{\it unital} if $A$ has a countable approximate unit. 
Note that if $A$ is separable, then $A$ is $\sigma$-unital. 
If $A$ is $\sigma$-unital, then there exists a positive element $s\in A$ such that $\{s^{\frac{1}{n}}\}_{n\in\mathbb{N}}$ is an approximate unit. 
Such a positive element $s$ is called {\it strictly positive} in $A$. 
We denote by  $A^{\sim}$ the unitization algebra of $A$. 
The \textit{multiplier algebra}, denoted by $M(A)$, of $A$ is the largest unital C$^*$-algebra that contains $A$ as an essential ideal. 
If $\alpha$ is an automorphism of $A$, then $\alpha$ extends uniquely to an automorphism of $M(A)$. We denote it by the same symbol $\alpha$ for simplicity. 

For a unitary element $u$ in $M(A)$, define an automorphism $\mathrm{Ad}(u)$ of $A$ by $\mathrm{Ad}(u) (x)=uxu^*$ for $x\in A$. Such an 
automorphism is called an \textit{inner automorphism}. Let $\mathrm{Aut}(A)$ denote the automorphism group of $A$, which is equipped with 
the topology of pointwise norm convergence. An automorphism $\alpha$ is said to be \textit{approximately inner} if $\alpha $ is in the closure of the inner automorphism group. 
We say that two automorphisms $\alpha$ and $\beta$ are \textit{approximately unitarily equivalent} if $\alpha\circ \beta^{-1}$ is approximately inner, and 
are \textit{outer conjugate} if there exist an automorphism $\gamma$ of $A$ and a unitary element $u$ in $M(A)$ such that 
$$
\alpha =\mathrm{Ad}(u)\circ \gamma\circ \beta \circ \gamma^{-1} .
$$ 

Let $F$ be a subset of $A$ and $\varepsilon >0$. A completely positive (c.p.) map $\varphi :A\to B$ is said to be \textit{$(F,\varepsilon)$-multiplicative} if 
$$
\| \varphi (xy) - \varphi (x)\varphi(y) \| < \varepsilon 
$$
for any $x,y\in F$. 
For c.p. maps $\varphi$, $\psi :A\to B$, we write $\varphi\sim_{F,\varepsilon}\psi$ if there exists a unitary element $u\in B^{\sim}$ such that 
$$
\| \varphi (x) - u\psi (x) u^* \| < \varepsilon 
$$
for any $x\in F$. 

We denote by $A_{+}$ the set of positive elements in $A$ and by $A_{+,1}$ the set of positive contractions in $A$. 
A \textit{trace} on $A$ is a map $\tau$ of $A_{+}$ to $[0,\infty ]$ 
such that $\tau (\lambda a)=\lambda\tau (a)$, $\tau (a+b)=\tau (a)+\tau (b)$ and $\tau (x^*x)=\tau (xx^*)$ for any $a,b\in A_{+}$, $\lambda \geq 0$ and $x\in A$. 
For a trace $\tau$ on $A$, let $\mathfrak{M}_{\tau}$ be a linear span of $\{a\in A_{+} \; | \; \tau (a)< \infty \}$ and 
$\mathfrak{N}_{\tau}:=\{x\in A\; :\; \tau (x^*x)<\infty \}$. 
Then $\mathfrak{M}_{\tau}$ and $\mathfrak{N}_{\tau}$ are ideals of $A$ and $\tau$ can be uniquely extended to a 
positive linear functional on $\mathfrak{M}_{\tau}$. 
A \textit{tracial state} is a trace which is a state. Every tracial state on $A$ extends uniquely to a tracial state on $M(A)$. 
We denote it by the same symbol for simplicity. 
We say that $\tau$ is \textit{densely defined} if $\mathfrak{M}_{\tau}$ is dense in $A$, and is \textit{lower semicontinuous} if 
$\{a\in A_{+} \; |\; \tau (a) \leq r \}$ is closed for any $r\in\mathbb{R}_{+}$. 
Let $T(A)$ denote the set of densely defined lower semicontinuous traces on $A$ and $T_1(A)$ the set of tracial states on $A$. 
For $\tau\in T(A)$, put $d_{\tau} (a)=\lim_{n\rightarrow \infty}\tau (a^{\frac{1}{n}})$ for 
$a\in A_{+}$. Then $d_{\tau}$ is a dimension function. 
We denote by $(\pi_{\tau}, H_{\tau})$  the Gelfand-Naimark-Segal (GNS) representation of $\tau\in T(A)$. 
Note that $H_{\tau}$ is the completion of the pre-Hilbert space $\mathfrak{N}_{\tau}$ with a pre-inner product 
$\langle \widehat{x},\widehat{y} \rangle = \tau (y^*x)$ for $x,y\in \mathfrak{N}_{\tau}$. 
The norm on $H_{\tau}$ is denoted by $\|\cdot\|_2$. 
Let $\tau$ be a lower semicontinuous densely defined trace on a $\sigma$-unital C$^*$-algebra $A$. 
We denote by $\mathrm{Ped}(A)$ the Pedersen ideal of $A$, which is a minimal dense ideal of $A$. 
Note that $\mathrm{Ped}(A)$ is contained in $\mathfrak{N}_{\tau}$ because $\mathfrak{N}_{\tau}$ is a dense ideal in $A$. 
There exists an approximate unit $\{h_n\}_{n\in\mathbb{N}}$ for $A$ contained in $\mathrm{Ped}(A)$. 
It easy to see that $\{\widehat{h_nx} \; | \; n\in\mathbb{N}, x\in \mathfrak{N}_{\tau} \}$ is dense in $H_{\tau}$. 
Indeed, the lower semicontinuity of $\tau$ implies that for any $x\in\mathfrak{N}_{\tau}$, we have 
$$
\| \widehat{x}-\widehat{h_n x} \|_2 =\tau (x^*(1-h_n)^2 x)^{\frac{1}{2}}\leq \tau (x^*(1-h_n) x)^{\frac{1}{2}} \to 0  
$$
as $n\to \infty$. 
If $\alpha$ is an automorphism of $A$ such that $\tau \circ \alpha =\lambda \tau$ for some $\lambda\in\mathbb{R}_{+}^{\times}$, 
then $\alpha$ can be uniquely extended to an automorphism $\tilde{\alpha}$ of $\pi_{\tau} (A)^{''}$. 

For $x,y\in A$, we write $[x,y]$ to mean the commutator $xy-yx$. 
We denote by $\mathbb{K}$ and $M_{n^\infty}$ for $n\in\mathbb{N}$ the C$^*$-algebra of compact operators on an infinite-dimensional 
separable Hilbert space  and the uniformly hyperfinite (UHF) algebra of type $n^{\infty}$, respectively. 
Let $\mathrm{Tr}_n$ for $n\in\mathbb{N}$ denote the (unnormalized) usual trace on $M_n(\mathbb{C})$ and $\mathrm{Tr}$ 
denote the usual trace on $\mathbb{K}$.

\subsection{Kirchberg's central sequence C$^*$-algebras}
We shall recall some properties of Kirchberg's central sequence C$^*$-algebras in \cite{Kir2} (see also \cite[Section 5]{Na2}). 
Let $\omega$ be a free ultrafilter on $\mathbb{N}$. For a $\sigma$-unital C$^*$-algebra $A$, set 
$$
c_{\omega}(A):=\{(x_n)_{n\in\mathbb{N}}\in \ell^{\infty}(\mathbb{N}, A)\; |\; \lim_{n \to \omega}\| x_n\| =0 \}, \; 
A^{\omega}:=\ell^{\infty}(\mathbb{N}, A)/c_{\omega}(A). 
$$
Let $B$ be a C$^*$-subalgebra of $A$. 
We identify $A$ and $B$ with the C$^*$-subalgebras of $A^\omega$ consisting of equivalence classes of 
constant sequences.  Put 
$$
A_{\omega}:=A^{\omega}\cap A^{\prime},\; \mathrm{Ann}(B,A^{\omega}):=\{(x_n)_n\in A^{\omega}\cap B^{\prime}\; |\; (x_n)_nb =0
\;\mathrm{for}\;\mathrm{any}\; b\in B \}.
$$
Then $\mathrm{Ann}(B,A^{\omega})$ is a closed ideal of $A^{\omega}\cap B^{\prime}$, and define 
$$
F(A):=A_{\omega}/\mathrm{Ann}(A,A^{\omega}).
$$
We call $F(A)$ the \textit{central sequence C$^*$-algebra} of $A$. A sequence $(x_n)_n$ is said to be 
\textit{central} if $\lim_{n\to \omega}\| [x_n,x] \| =0$ for all $x\in A$. A central sequence 
is a representative of an element in $A_{\omega}$. 
Since $A$ is $\sigma$-unital, $A$ has a countable approximate unit $\{h_n\}_{n\in\mathbb{N}}$. 
It is easy to see that $[(h_n)_n]$ is a unit in $F(A)$. If $A$ is unital, then $F(A)=A_{\omega}$. 
Note that $F(A)$ is isomorphic to $M(A)^\omega\cap A^{\prime}/\mathrm{Ann}(A,M(A)^{\omega})$ and ${A^{\sim}}_{\omega}/ \mathrm{Ann}(A,(A^{\sim})^{\omega})$. 
Indeed, for any $(y_n)_n\in M(A)^\omega\cap A^{\prime}$ (respectively, $(y_n)_n\in {A^{\sim}}_{\omega}$), $(y_nh_n)_n$ is a central sequence in $A$ and 
$[(y_n)_n]=[(y_nh_n)_n]$ in $M(A)^\omega\cap A^{\prime}/\mathrm{Ann}(A,M(A)^{\omega})$ (respectively, in ${A^{\sim}}_{\omega}/ \mathrm{Ann}(A,(A^{\sim})^{\omega})$). 
Let $h$ be a full positive element in $A$, and define a map $\eta$ from $F(A)$ to $F(\overline{hAh})$ by $\eta ([(x_n)_n])=[(h^{\frac{1}{n}}x_nh^{\frac{1}{n}})_n]$. 
Then $\eta$ is an isomorphism from $F(A)$ onto $F(\overline{hAh})$. In particular, $F(A\otimes\mathbb{K})$ is isomorphic to $F(A)$. 
If $\alpha$ is an automorphism of $A$, $\alpha$ induces natural automorphisms of $A^{\omega}$, $A_{\omega}$ and $F(A)$. 
We denote them by the same symbol $\alpha$ for simplicity. 

There exists a natural homomorphism $\rho$ from $F(A)\otimes_{\mathrm{max}}A$ to $A^{\omega}$ such that 
$$
\rho ([(x_n)_n]\otimes x) = (x_nx)_n 
$$
for any $[(x_n)_n]\in F(A)$ and $x\in A$. 
For a projection $p$ in $F(A)$, let $A^{\omega}_p$ be a hereditary subalgebra of $A^{\omega}$ generated by 
$\rho (pF(A)p\otimes_{\max} A) $. It can be easily checked that 
$$
A^{\omega}_p = \overline{\rho (p\otimes s)A^{\omega} \rho (p\otimes s)}
$$
where $s$ is a strictly positive element in $A$. 

Since $\omega$ is a (free) ultrafilter, for any $\tau\in T_1(A)$, we can define a tracial state $\tau_{\omega}$ on $A^{\omega}$ by 
$\tau_{\omega} ((a_n)_n)= \lim_{n\to \omega} \tau (a_n)$ for any $(a_n)_n\in A^{\omega}$. 
We shall show that $\tau_{\omega}$ is well-defined on $F(A)$. 

\begin{pro}
Let $A$ be a $\sigma$-unital C$^*$-algebra, and let $\tau$ be a tracial state on $A$. Define $\tau_{\omega}: F(A)\to \mathbb{C}$ by 
$$
\tau_{\omega} ([(x_n)_n]) = \lim_{n\to\omega} \tau(x_n)
$$
for any $[(x_n)_n]\in F(A)$. Then $\tau_{\omega}$ is well-defined. In particular, $\tau_{\omega}$ is a tracial state on $F(A)$. 
\end{pro}
\begin{proof}
It suffices to show that if $(x_n)_n\in \mathrm{Ann}(A,A^{\omega})$, then $\lim_{n\to\omega} \tau (x_n)=0$. We may assume that 
$\| x_n\| \leq 1$ for any $n\in\mathbb{N}$. Let $\{h_n\}_{n\in\mathbb{N}}$ be an approximate unit for $A$ and $\varepsilon >0$. 
There exists a natural number $N$ such that 
$$
|1 -\tau (h_N) | < \frac{\varepsilon}{2}
$$
because $\lim_{n\to \infty}\tau (h_n)=1$. Since $\lim_{n\to \omega} \| x_nh_{N} \| =0$, there exists $X\in\omega$ such that 
$$
|\tau (x_nh_N)| < \frac{\varepsilon}{2} 
$$
for any $n\in X$. Hence we have 
\begin{align*}
|\tau (x_n) | 
& \leq |\tau (x_n) -\tau (x_n h_N)| + |\tau (x_n h_N)| < |\tau((1-h_N)^{\frac{1}{2}}x_n(1-h_N)^{\frac{1}{2}})| +\frac{\varepsilon}{2} \\
& \leq |1- \tau (h_{N})| +\frac{\varepsilon}{2} < \varepsilon
\end{align*}
for any $n\in X$. Therefore $\lim_{n\to\omega} \tau (x_n)=0$.
\end{proof}

For a semifinite von Neumann algebra $M$ with separable predual, set 
$$
\mathscr{C}_{\omega} (M):= \{(x_n)_{n\in\mathbb{N}}\in \ell^{\infty}(\mathbb{N}, M)\; | \; [x_n,y] \rightarrow  0 \; *\text{-strongly as}\; n\rightarrow \omega \; \text{for any}\; y\in M \}
$$
and 
$$
\mathscr{T}_{\omega} (M):= \{(x_n)_{n\in\mathbb{N}}\in \ell^{\infty}(\mathbb{N}, M)\; | \; x_n \rightarrow  0 \; *\text{-strongly as}\; n\rightarrow \omega \}.
$$
Then $\mathscr{C}_{\omega} (M)$ is a C$^*$-subalgebra of  $\ell^{\infty}(\mathbb{N}, M)$ and $\mathscr{T}_{\omega}(M)$ is a closed ideal of $\mathscr{C}_{\omega}(M)$.  
Define 
$$
M_{\omega}:= \mathscr{C}_{\omega}(M)/\mathscr{T}_{\omega}(M).
$$
Note that $M_{\omega}$ coincides with the asymptotic centralizer of $M$ in \cite{C1}, and hence $M_{\omega}$ is a finite von Neumann algebra. 
(Indeed, \cite[Lemma XIV.3.4]{Tak} and some arguments on semifinite von Neumann algebras show this fact.) 
Let $p$ be a projection in $M$ with central support $1$, and define a map $\tilde{\eta}$ from $M_{\omega}$ to $(pMp)_{\omega}$ by $\tilde{\eta} ((x_n)_n) =(px_np)_n$. 
Then $\tilde{\eta}$ is an isomorphism from $M_{\omega}$ onto $(pMp)_{\omega}$ by \cite[Lemma 2.11]{C1} (see also \cite[Lemma 2.8]{TM}) and \cite[Proposition 2.10]{TM}. 
If $\alpha$ is an automorphism of a semifinite von Neumann algebra $M$, $\alpha$ induces a natural automorphism of $M_{\omega}$. 
We denote it by the same symbol $\alpha$ for simplicity. 
Note that if $\tau$ is a bounded trace, then the following proposition is clear. 

\begin{pro}\label{pro:clear-inclusion}
Let $A$ be a $\sigma$-unital C$^*$-algebra, and let $\tau$ be a lower semicontinuous densely defined trace on $A$. 
Then the inclusion map from $A$ to $\pi_{\tau} (A)^{''}$ induces a homomorphism $\varrho_A$ 
from $F(A)$ to $\pi_{\tau} (A)^{''}_{\omega}$. 
\end{pro}
\begin{proof}
Let $\{h_m\}_{m\in\mathbb{N}}$ be an approximate unit for $A$ contained in $\mathrm{Ped}(A)$.  
First, we shall show that if $(x_n)_n\in \mathrm{Ann}(A,A^{\omega})$, then $(\pi_{\tau}(x_n))_n\in \mathscr{T}_{\omega}(\pi_{\tau} (A)^{''})$. 
For any $m\in\mathbb{N}$ and $x\in \mathfrak{N}_{\tau}$, we have 
$$
\| \pi_{\tau}(x_n) \widehat{h_mx} \|_{2}+ \| \pi_{\tau}(x_n^*) \widehat{h_mx} \|_{2}\leq \| x_n h_m \|\cdot \| \widehat{x}\|_{2} +  \| x_n^* h_m \|\cdot \| \widehat{x}\|_{2} \to 0 
$$
as $n \to \omega$. Hence we see that $(\pi_{\tau}(x_n))_n\in \mathscr{T}_{\omega}(\pi_{\tau} (A)^{''})$. 

Let $(x_n)_n$ be a central sequence of contractions in $A$. If we prove $(\pi_{\tau}(x_n))_n \in \mathscr{C}_{\omega}(\pi_{\tau} (A)^{''})$, 
then we obtain the conclusion. 
Let $\varepsilon >0$ and $y\in \pi_{\tau}(M)^{''}$. 
It suffices to show that for any $m\in\mathbb{N}$ and $x\in\mathfrak{N}_{\tau}$, there exists $X\in \omega$ such that 
$$
\|[\pi_{\tau}(x_n), y]\widehat{h_mx} \|_2 + \|[\pi_{\tau}(x_n), y]^*\widehat{h_mx} \|_2 < \varepsilon 
$$
for any $n\in X$. We may assume that $y$ is a contraction and $\widehat{x}\neq 0$. (Note that we also have $x\neq 0$.) 
By the Kaplansky density theorem, there exists a contraction $z\in A$ such that 
$$
\| (y-\pi_{\tau} (z))\widehat{h_m} \|_2 + \|(y^*-\pi_{\tau}(z^*))\widehat{h_m} \|_{2} < \frac{\varepsilon}{12\| x\|}. 
$$
Since $(x_n)_n$ is a central sequence in $A$, there exists  $X\in \omega$ such that 
$$
\| [x_n, z] \| < \frac{\varepsilon}{6\|\widehat{x}\|_{2}}, \quad \| [x_n, h_m] \| < \frac{\varepsilon}{12\|\widehat{x}\|_2}
$$
for any $n\in X$. 
Then we have 
\begin{align*}
\|[\pi_{\tau}(x_n), y]\widehat{h_mx} \|_2 
&\leq \| (\pi_{\tau} (x_n) y -\pi_{\tau}(x_nz) ) \widehat{h_mx}\|_2 + \| (\pi_{\tau}(x_n z) -y\pi_{\tau}(x_n)) \widehat{h_mx}\|_2 \\ 
&\leq \|x_n \|\cdot\| (y-\pi_{\tau}(z))\widehat{h_m} \|_2 \cdot \| x\| + \| (\pi_{\tau}(x_n z) -y\pi_{\tau}(x_n))\widehat{h_mx}\|_2 \\ 
&< \frac{\varepsilon}{12}+ \|\pi_{\tau}([x_n, z]) \widehat{h_mx} \|_2 + \| (\pi_{\tau}(zx_n) - y \pi_{\tau}(x_n))\widehat{h_mx}\|_2 \\
&\leq \frac{\varepsilon}{12}+ \| [x_n, z] \|\cdot\| h_m\|\cdot \|\widehat{x}\|_2 + \| (\pi_{\tau}(zx_n) - y \pi_{\tau}(x_n))\widehat{h_mx}\|_2 \\ 
&< \frac{\varepsilon}{4}+ \|\pi_{\tau}(z[x_n,h_m]) \widehat{x}\|_2 + \|\pi_{\tau}(zh_mx_n) \widehat{x}- y \pi_{\tau}(x_n)\widehat{h_mx}\|_2 \\
&\leq \frac{\varepsilon}{4}+ \| z\|\cdot\| [x_n, h_m] \|\cdot\|\widehat{x}\|_2+ \| (\pi_{\tau}(zh_mx_n) - y \pi_{\tau}(x_nh_m))\widehat{x}\|_2 \\ 
&< \frac{\varepsilon}{3}+\|\pi_{\tau}(zh_mx_n) \widehat{x}- y \pi_{\tau}(h_mx_n)\widehat{x}\|_2 +\|y \pi_{\tau}([h_m, x_n])\widehat{x}\|_2 \\
&\leq \frac{\varepsilon}{3}+ \| (\pi_{\tau}(z)-y)\widehat{h_m}\|_2\cdot \| x_n\|\cdot \| x \| + \| y\|\cdot\|[h_m,x_n]\|\cdot \| \widehat{x}\|_2 \\
& <\frac{\varepsilon}{2} 
\end{align*}
for any $n\in X$. 
Similar arguments show
$$
\|[\pi_{\tau}(x_n), y]^*\widehat{h_mx} \|_2 < \frac{\varepsilon}{2}
$$
for any $n\in X$. 
Therefore the proof is complete. 
\end{proof}

It is easy to see that if $\tau$ is a bounded faithful trace on $A$, then 
$\mathrm{ker}(\varrho_{A})= \{[(x_n)_n]\in F(A)\; | \; \lim_{n\to\omega} \tau (x_n^*x_n)=0 \}$. 
Moreover, the same proof as  \cite[Theorem 3.3]{KR} shows that $\varrho_A$ is surjective. (See also \cite[Lemma 2.1]{S}.) 
Using this fact, we show the following proposition.

\begin{pro}\label{pro:sato}
Let $A$ be a $\sigma$-unital C$^*$-algebra, and let $\tau$ be a faithful lower semicontinuous densely defined trace on $A$. Assume that 
$\mathrm{Ped}(A)$ contains a full positive element. Then $\varrho_A$ is surjective. 
\end{pro}
\begin{proof}
Let $h$ be a full positive element in $\mathrm{Ped}(A)$. Then $\overline{hAh}$ is a 
C$^*$-algebra with no unbounded traces by the same argument as in the proof of \cite[Proposition 5.2]{Na2}. 
In particular, $\tau|_{\overline{hAh}}$ is a bounded faithful trace on $\overline{hAh}$. 
Let $p$ be a support projection of $\pi_{\tau} (h)$ in $\pi_\tau (A)^{''}$. Note that $\pi_{\tau}(h^{\frac{1}{n}})$ converges to 
$p$ in the strong$^*$ topology as $n\to \infty$. Since $h$ is full in $A$, $p$ has central support $1$. 
Hence we have the following commutative diagram: 
\[\begin{CD}
       F(A)              @>\varrho_{A}>>  \pi_{\tau} (A)^{''}_{\omega} \\
            @VV\eta V                               @VV\tilde{\eta}V  \\                 
      F(\overline{hAh})  @>\varrho_{\overline{hAh}}>> (p\pi_{\tau} (A)^{''}p)_{\omega}
\end{CD} \]
where $\eta$ and $\tilde{\eta}$ are  standard isomorphisms from $F(A)$ onto $F(\overline{hAh})$ and 
from $\pi_{\tau} (A)^{''}_{\omega}$ onto $(p\pi_{\tau} (A)^{''}p)_{\omega}$, respectively. 
Since $\varrho_{\overline{hAh}}$ is surjective,  we see that $\varrho_A$ is surjective.
\end{proof}

\begin{rem}\label{rem:trace-ideal} 
Let $A$ be a $\sigma$-unital simple C$^*$-algebra, and let $\tau$ be a lower semicontinuous densely defined trace on $A$. 
Then $\mathrm{Ped}(A)$ contains a full positive element since every non-zero positive element is full. 
Moreover, it can be easily checked that $\mathrm{ker}(\varrho_{A})= \{[(x_n)_n]\in F(A)\; | \;  \lim_{n\to\omega} \| \pi_{\tau}(x_n)\hat{h}\|_2=0 \}$ 
for some $h\in \mathrm{Ped}(A)_{+}\setminus \{0\}$. Assume that $\tau$ is a unique (up to scalar multiple) trace. 
Then we can define a tracial state $\tau_{\omega}$ on $F(A)$ by $\tau_{\omega} ([(x_n)_n]):= \lim_{n\to\omega} \tau (x_nh)/\tau (h)$. 
Note that if $[(x_n)_n]\in\mathrm{ker}(\varrho_A)$, then $\tau_{\omega}([(x_n)_n])=0$. 
\end{rem}

\subsection{Matui and Sato's result}

We shall remark that some arguments in \cite{MS2} work for non-unital C$^*$-algebras. 
It is important to consider property (SI). 
For $a,b\in A_{+}$, we say that $a$ is \textit{Cuntz smaller than} $b$, written $a\precsim b$, if there exists a 
sequence $\{x_n\}_{n\in\mathbb{N}}$ of $A$ such that $\| x_n^*bx_n-a\|\rightarrow 0$. 

\begin{Def}\label{def:si}
Let $A$ be a C$^*$-algebra. 
\ \\
(1) Assume that $T_1(A)$ is a non-empty compact set. 
We say that $A$ has \textit{property (SI)} if for any central sequences $(a_n)_n$ and $(b_n)_n$ of positive contractions 
in $A$ satisfying 
$$
\lim_{n\to\omega}\max_{\tau\in T_1(A)}\tau (a_n)=0,\;\inf_{m\in\mathbb{N}}\lim_{n\to\omega}\min_{\tau\in T_1(A)}\tau (b_n^m)>0,
$$
there exists a central sequence $(s_n)_n$ in $A$ such that 
$$
\lim_{n\to\omega}\| s_n^*s_n-a_n\| =0, \; \lim_{n\to\omega}\| b_ns_n-s_n\| =0. 
$$
(2) Let $S$ be a subset of $T_1(A)$. 
We say that $A$ has \textit{strict comparison} (respectively, \textit{strict comparison with respect to $S$}) if for any $k\in\mathbb{N}$,  
$a,b\in M_k(A)_{+}$ with $d_{\tau\otimes\mathrm{Tr}_k}(a)< d_{\tau\otimes\mathrm{Tr_k}} (b)$ for any 
$\tau \in T_1(A)$ (respectively, for any $\tau\in S$) implies $a\precsim b$. 
\end{Def}

Note that strict comparison in the definition above is different from almost unperforation of the Cuntz semigroup $\mathrm{Cu}(A)$. 
Essentially the same proofs as \cite[Lemma 4.7]{MS2} and \cite[Proposition 4.8]{MS2} show the following theorem. See also \cite[Propostion 3.3]{MS3} and 
\cite[Theorem 4.1]{BBSTWW}. 

\begin{thm}\label{thm:Matui-Sato}
(Matui-Sato) \ \\
Let $A$ be a simple separable infinite-dimensional nuclear C$^*$-algebra with finitely many extremal tracial states and no unbounded traces. Assume that 
$A$ has property (SI). Then: \ \\
(i) For any tracial state $\sigma$ on $F(A)$, there exists a tracial state $\tau$ on $A$ such that $\sigma = \tau_{\omega}$. \ \\
(ii) If $a$ and $b$ are positive elements in $F(A)$ satisfying $d_{\tau_{\omega}}(a)< d_{\tau_{\omega}} (b)$ for any $\tau \in T_1(A)$, then there exists an element $r\in F(A)$ such that 
$r^*br=a$. Moreover, $F(A)$ has strict comparison. 
\end{thm}

\subsection{Razak-Jacelon algebra}

Let $\mathcal{W}$ be the Razak-Jacelon algebra studied in \cite{J}, which has trivial $K$-groups and a unique tracial state $\tau$ and no unbounded traces. 
The Razak-Jacelon algebra $\mathcal{W}$ is constructed as an inductive limit C$^*$-algebra of Razak's building block in \cite{Raz}, that is,  
$$
A(n,m)= \left\{f\in C([0,1])\otimes M_m(\mathbb{C}) \ | \
\begin{array}{cc} 
f(0)=\mathrm{diag}(\overbrace{c,..,c}^k,0_{n}),
f(1)=\mathrm{diag}(\overbrace{c,..,c}^{k+1}), \\
c\in M_n(\mathbb{C})
\end{array} 
\right\},
$$
where $n$ and $m$ are natural numbers with $n|m$ and $k:=m/n-1$. 
Let $\mathcal{O}_{2}$ denote the Cuntz algebra generated by $2$ isometries 
$S_1$ and $S_2$. 
For every $\lambda_1,\lambda_2\in\mathbb{R}$ there exists by universality a one-parameter 
automorphism group $\alpha$ of $\mathcal{O}_2$ given by $\alpha_t (S_j)=e^{it\lambda_{j}}S_j$. 
Kishimoto and Kumjian showed that if 
$\lambda_{1}$ and $\lambda_{2}$ are all non-zero, of the same sign and 
$\lambda_1$ and $\lambda_2$ generate $\mathbb{R}$ as a closed subgroup, then 
$\mathcal{O}_2\rtimes_{\alpha}\mathbb{R}$ is a simple stably projectionless C$^*$-algebra 
with unique (up to scalar multiple) densely defined lower semicontinuous trace in \cite{KK1} and \cite{KK2}. 
Moreover, Robert \cite{Rob} showed that $\mathcal{W}\otimes \mathbb{K}$ is isomorphic to 
$\mathcal{O}_2\rtimes_{\alpha}\mathbb{R}$ for some $\lambda_1$ and $\lambda_2$. 
(See also \cite{Dean}.) 

By the uniqueness of traces on $\mathcal{W}\otimes\mathbb{K}$, for any automorphism $\alpha$ of $\mathcal{W}\otimes\mathbb{K}$, 
there exists a positive real number $\lambda (\alpha)$ such that 
$(\tau\otimes\mathrm{Tr}) \circ \alpha =\lambda(\alpha) \tau\otimes\mathrm{Tr}$. We say that $\alpha$ is a \textit{trace scaling automorphism} if $\lambda (\alpha)\neq 1$. 
The following theorem is an immediate consequence of Razak's classification theorem (see also \cite{Rob}). 

\begin{thm}\label{thm:Razak} (Razak) \ \\
(i) Let $A$ be a simple unital approximately finite-dimensional (AF) algebra with a unique tracial state. Then $A\otimes\mathcal{W}$ is isomorphic to 
$\mathcal{W}$. 
\ \\
(ii) Every automorphism of $\mathcal{W}$ is approximately inner. \ \\
(iii) Let $\alpha$ and $\beta$ be automorphisms of $\mathcal{W}\otimes\mathbb{K}$. 
Then $\alpha$ and $\beta$ are approximately unitarily equivalent if and only if $\lambda (\alpha)=\lambda (\beta)$. \ \\
(iv) For any $\lambda\in\mathbb{R}^{\times}_{+}$, there exists an automorphism $\alpha$ of $\mathcal{W}\otimes\mathbb{K}$ such that 
$\lambda(\alpha)=\lambda$. 
\end{thm}

Note that $\mathcal{W}$ is UHF-stable by (i) in the theorem above, and hence $\mathcal{W}$ is $\mathcal{Z}$-stable. 
Hence $\mathcal{W}$ has strict comparison and property (SI) (see \cite{Ror}, \cite{MS} and \cite{Na2}). 

\section{Stable uniqueness theorem}\label{sec:stable}

In this section we shall recall some results in \cite{EN} and reformulate for our purpose. 
Let $\Omega$ be a compact metrizable space. 
The following proposition is based on the results of \cite{DE1}, \cite{DE2}, \cite{EllK} and \cite{G}. 

\begin{pro}\label{pro:EN8.2} (cf. \cite[Proposition 8.2]{EN}) \ \\
Let $A$ be a separable non-unital C$^*$-algebra and $B$ a separable C$^*$-algebra, and let $\sigma$ be a full homomorphism 
from $A$ to $B$. 
Suppose that $\varphi$ and $\psi$ are nuclear homomorphisms from $C(\Omega)\otimes A$ to $B$ with $[\varphi]= [\psi]$ in $KK_{\mathrm{nuc}}(C(\Omega)\otimes A,B)$. Then for any 
finite subsets $F_1\subset C(\Omega)$, $F_2\subset A$ and $\varepsilon>0$, there exist $m\in \mathbb{N}$, $z_1,z_2,...,z_m\in \Omega$ and 
a unitary element $u$ in $M_{m^2+1}(B)^{\sim}$ such that 
\begin{align*}
\| u^{*}(\varphi(f\otimes a) \oplus & \overbrace{\bigoplus_{k=1}^m f(z_k)\sigma (a) \oplus \cdots \oplus \bigoplus_{k=1}^m f(z_k)\sigma (a)}^m) u \\
& - \psi(f\otimes a)\oplus \overbrace{\bigoplus_{k=1}^m f(z_k)\sigma (a) \oplus \cdots \oplus\bigoplus_{k=1}^m f(z_k)\sigma (a) }^m\| < \varepsilon 
\end{align*}
for any $f\in F_1$ and $a\in F_2$. 
\end{pro}
\begin{proof}
Choose a dense subset $\{z_k \; |\; k\in\mathbb{N}\}\subset \Omega$. Let $\gamma$ be a homomorphism from $C(\Omega)\otimes A$ to $M(B\otimes\mathbb{K}\otimes\mathbb{K})$ such that 
$$
\gamma (f\otimes a) = \sigma (a) \otimes \sum_{k=1}^{\infty}f(z_k) e_{kk} \otimes 1 
$$
for any $f\in C(\Omega)$ and $a\in A$, where $\{e_{ij}\}_{i,j\in\mathbb{N}}$ is the standard matrix units of $\mathbb{K}$. 
Then we see that $\gamma$ is purely large as an extension by \cite[Theorem 17 (iii)]{EllK}. 
The same arguments as in the proofs of \cite[Lemma 8.1]{EN} and \cite[Proposition 8.2]{EN} show that there exists a sequence $\{u_n\}_{n\in\mathbb{N}}$ of unitaries in 
$(B\otimes\mathbb{K})^{\sim}$ such that 
$$
\| u_n (\varphi (f\otimes a) \oplus \gamma (f\otimes a) ) u_n^* - \psi (f\otimes a)\oplus \gamma (f\otimes a ) \| \to 0 
$$
as $n\to \infty$ for any $f\in C(\Omega)$ and $a\in A$. 
For any $m\in\mathbb{N}$, let  
$$
e_{m}:= 1 \oplus \sum_{k=1}^{m} e_{kk} \otimes \sum_{k=1}^{m} e_{kk} \in B^{\sim}\oplus B^{\sim}\otimes\mathbb{K}\otimes\mathbb{K} \subset B^{\sim}\otimes \mathbb{K}. 
$$
For any $n\in\mathbb{N}$, we have $\| [e_m , u_n]\| \to 0$ as $m\to \infty$. Hence for sufficiently large $m$, $e_mu_ne_m$ is close to a unitary element in 
$e_m (B\otimes \mathbb{K})^{\sim}e_m \cong M_{m^2+1} (B)^{\sim}$. 

Therefore choose a sufficient large $n\in\mathbb{N}$, and then choose a sufficiently large $m\in\mathbb{N}$, we can find a unitary element $u$ in $M_{m^2+1} (B)^{\sim}$ satisfying 
the conclusion of the proposition. 
\end{proof}
In order to obtain a stable uniqueness theorem for $(G,\delta )$-multiplicative maps, 
we need to consider homomorphisms from $A$ to $\prod B_n / \bigoplus B_n$ for some C$^*$-algebras $A$ and $B_n$. 
We can avoid the assumption of separability in the proposition above by Blackadar's technique (see \cite[II.8.5]{Bla} and \cite[Lemma 8.4]{EN}). 
It is useful to consider the following for the fullness. 

\begin{Def} (cf. \cite[Definition 8.7]{EN}) \ \\
Let $L$ be a map from $A_{+,1}\setminus\{ 0\} \times (0,1)$ to $\mathbb{N}$ and $N$ a map from $A_{+,1}\setminus \{0\} \times (0,1)$ to $(0,\infty)$. 
A homomorphism $\varphi$ from $A$ to $B$ is said to be \textit{$(L,N)$-full} if for any $\varepsilon\in (0,1)$, $a\in A_{+,1}\setminus\{0\}$ and $b\in B_{+,1}$, there exist 
elements $x_1,x_2,...,x_{L(a,\varepsilon)}$ in $B$ such that 
$$
\| x_i \| \leq N(a,\varepsilon)
$$ 
for any $i=1,2,...,L(a,\varepsilon)$ and 
$$
\| b - \sum_{i=1}^{L(a,\varepsilon)} x_i \varphi (a) x_{i}^* \| < \varepsilon .
$$
\end{Def}

The following proposition is a variant of \cite[Proposition 8.12]{EN}. For finite sets $F_1$ and $F_2$, let $F_1\odot F_2:=\{a\otimes b \; |\; a\in F_1, b\in F_2\}$. 

\begin{pro}\label{pro:stable-uniqueness}
Let $A$ be a separable non-unital nuclear C$^*$-algebra that is $KK$-equivalent to $\{0\}$, and 
let $L: A_{+,1}\setminus \{0\} \times (0,1 )\to \mathbb{N}$ and $N: A_{+,1}\setminus \{0\}\times (0,1)\to (0,\infty)$ be maps. 
Then for any finite subsets $F_1\subset C(\Omega)$, $F_2\subset A$ and $\varepsilon>0$, there exist finite subsets $G_1\subset C(\Omega)$, $G_2\subset A$, $m\in\mathbb{N}$, 
and $\delta >0$ such that the following holds. Let $B$ be a C$^*$-algebra. For any contractive ($G_1\odot G_2, \delta$)-multiplicative maps $\varphi, \psi:C(\Omega)\otimes A\to 
B$ and an $(L,N)$-full homomorphism $\sigma : A\to B$, there exist a unitary element in $u$ in $M_{m^2+1}(B)^{\sim}$ and $z_1,z_2,...,z_m\in\Omega$ such that 
\begin{align*}
\| u^{*}(\varphi(f\otimes a) \oplus & \overbrace{\bigoplus_{k=1}^m f(z_k)\sigma (a) \oplus \cdots \oplus \bigoplus_{k=1}^m f(z_k)\sigma (a)}^m) u \\
& - \psi(f\otimes a)\oplus \overbrace{\bigoplus_{k=1}^m f(z_k)\sigma (a) \oplus \cdots \oplus \bigoplus_{k=1}^m f(z_k)\sigma (a)}^m\| < \varepsilon 
\end{align*}
for any $f\in F_1$ and $a\in F_2$. 
\end{pro}
\begin{proof}
Let finite subsets $F_1\subset C(\Omega)$, $F_2\subset A$ and $\varepsilon>0$. On the contrary, suppose that the proposition were false for $F_1$, $F_2$ and $\varepsilon$. 
Then for any $n\in\mathbb{N}$, there exist a C$^*$-algebra $B_n$, contractive c.p. maps $\varphi_n, \psi_n : C(\Omega)\otimes A \to B_n$ and an $(L,N)$-full homomorphism 
$\sigma_n :A \to B_n$ such that 
$$
\| \varphi_n (xy)- \varphi_n(x)\varphi_n(y) \| \to 0, \quad  \| \psi_n (xy)- \psi_n(x)\psi_n(y) \| \to 0 
$$
as $n\to\infty$ for any $x,y\in C(\Omega)\otimes A$ and there exist no unitaries in $M_{n^2+1}(B_n)^{\sim}$ and $n$ elements in $\Omega$ satisfying the conclusion of the proposition. 

Define homomorphisms $\Phi$ and $\Psi$ from $C(\Omega) \otimes A$ to $\prod B_n/\bigoplus B_n$ by
$$ 
\Phi (x) := (\varphi_n(x))_n, \quad \Psi (x) := (\psi_n(x))_n
$$ 
for any $x\in C(\Omega)\otimes A$, and define a homomorphism $\Sigma$ from $A$ to $\prod B_n/\bigoplus B_n$ 
by 
$$
\Sigma (a):= (\sigma_n(a))_n
$$ for any $a\in A$. Since $\sigma_n$ is $(L,N)$-full for any $n\in\mathbb{N}$, $\Sigma$ is full in $\prod B_n/\bigoplus B_n$. 
By \cite[Lemma 8.4]{EN}, there exists a separable C$^*$-subalgebra $B$ of $\prod B_n/\bigoplus B_n$ such that 
$$
\Phi (C(\Omega)\otimes A), \Psi  (C(\Omega)\otimes A), \Sigma (A) \subset B
$$
and $\Sigma$ is full in $B$. 

Since $C(\Omega)\otimes A$ is non-unital, nuclear and $KK$-equivalent to $\{0 \}$, it follows from Proposition \ref{pro:EN8.2} that there exist 
$m\in \mathbb{N}$, $z_1,z_2,...,z_m\in \Omega$ and 
a unitary element $U$ in $M_{m^2+1}(B)^{\sim}$ such that 
\begin{align*}
 \| U^{*}(\Phi(f\otimes a) \oplus & \overbrace{\bigoplus_{k=1}^m f(z_k)\Sigma (a) \oplus \cdots \oplus\bigoplus_{k=1}^m f(z_k)\Sigma (a) }^m) U \\
& - \Psi(f\otimes a)\oplus \overbrace{\bigoplus_{k=1}^m f(z_k)\Sigma (a) \oplus \cdots \oplus \bigoplus_{k=1}^m f(z_k)\Sigma (a)}^m\| < \varepsilon 
\end{align*}
for any $f\in F_1$ and $a\in F_2$. 
It is easy to see that $U$ can be lifted to a unitary element $(u_n)_{n\in\mathbb{N}}$ in $\prod M_{m^2+1} (B_n)^{\sim}$. 
Note that we may assume $u_n=b_n+1$ for some $b_n\in M_{m^{2}+1}(B_n)$. 
For sufficiently large $n\in \mathbb{N}$, 
we have 
\begin{align*}
 \| u_n^{*}(\varphi_n(f\otimes a) \oplus & \overbrace{\bigoplus_{k=1}^m f(z_k)\sigma_n (a) \oplus \cdots \oplus \bigoplus_{k=1}^m f(z_k)\sigma_n (a)}^m) u_n \\
& - \psi_n(f\otimes a)\oplus \overbrace{\bigoplus_{k=1}^m f(z_k)\sigma_n (a) \oplus \cdots \oplus \bigoplus_{k=1}^m f(z_k)\sigma_n (a)}^m\| < \varepsilon 
\end{align*}
for any $f\in F_1$ and $a\in F_2$. 
Take elements $z_{m+1},...,z_{n}\in \Omega$, and let $u$ be a trivial unitary extension in $M_{n^2+1}(B_n)^{\sim}$ from $u_n$. Then 
\begin{align*}
 \| u^{*}(\varphi_n(f\otimes a) \oplus & \overbrace{\bigoplus_{k=1}^n f(z_k)\sigma_n (a) \oplus \cdots \oplus\bigoplus_{k=1}^n f(z_k)\sigma_n (a) }^n) u \\
& - \psi_n(f\otimes a)\oplus \overbrace{\bigoplus_{k=1}^n f(z_k)\sigma_n (a) \oplus \cdots \oplus \bigoplus_{k=1}^n f(z_k)\sigma_n (a)}^n\| < \varepsilon 
\end{align*}
for any $f\in F_1$ and $a\in F_2$. This is a contradiction. Therefore the proof is complete. 
\end{proof}

The following lemma is an analogous lemma of \cite[Lemma 2.2 (iii)]{KR2}. 

\begin{lem}\label{lem:h-LN-full}
Let $A$ be a C$^*$-algebra, and let $L: A_{+,1}\setminus \{0\} \times (0,1 )\to \mathbb{N}$ and $N: A_{+,1}\setminus \{0\}\times (0,1)\to (0,\infty)$ 
be maps. Then there exist maps $L^{\prime}: A_{+,1}\setminus \{0\} \times (0,1 )\to \mathbb{N}$ and $N^{\prime}: A_{+,1}\setminus \{0\}\times (0,1)\to (0,\infty)$ such that 
the following holds. Let $B$ be a C$^*$-algebra, and let $C$ be a hereditary subalgebra of $B$.
If $\sigma$ is a homomorphism from $A$ to $C\subset B$ such that for any $\varepsilon\in (0,1)$, $a\in A_{+,1}\setminus\{0\}$ and 
$b\in C_{+,1}$, there exist 
elements $x_1,x_2,...,x_{L(a,\varepsilon)}$ in $B$ such that 
$$
\| x_i \| \leq N(a,\varepsilon)
$$ 
for any $i=1,2,...,L(a,\varepsilon)$ and 
$$
\| b - \sum_{i=1}^{L(a,\varepsilon)} x_i \sigma (a) x_{i}^* \| < \varepsilon ,
$$
then $\sigma$ is $(L^{\prime},N^{\prime})$-full in $C$. 
\end{lem}
\begin{proof}
For any $a\in A_{+,1}\setminus\{0 \}$ and $\varepsilon\in (0,1)$, let 
$$
L^{\prime}(a, \varepsilon) := L(a^2, \varepsilon), \quad N^{\prime}(a, \varepsilon) := N(a^2, \varepsilon).
$$
Then $L^{\prime}$ and $N^{\prime}$ have the desired property. Indeed, let $\sigma$ be a homomorphism from $A$ to $C\subset B$ satisfying the assumption above. 
For any $b\in C_{+,1}$ and $\varepsilon\in (0,1)$, there exist 
elements $x_1,x_2,...,x_{L^{\prime}(a,\varepsilon)}$ in $B$ such that 
$\| x_i \| \leq N^{\prime}(a,\varepsilon)$ for any $i=1,2,...,L^{\prime}(a,\varepsilon)$ and 
$$
\| b^{\frac{1}{2}} - \sum_{i=1}^{L^{\prime}(a,\varepsilon)} x_i \sigma (a^2) x_{i}^* \| < \varepsilon .
$$
We have 
$$
\| b - \sum_{i=1}^{L^{\prime}(a,\varepsilon)} b^{\frac{1}{4}}x_i \sigma(a^{\frac{1}{2}}) \sigma (a) \sigma(a^{\frac{1}{2}}) x_{i}^* b^{\frac{1}{4}} \| \leq 
\| b^{\frac{1}{2}} - \sum_{i=1}^{L^{\prime}(a,\varepsilon)} x_i \sigma (a^2) x_{i}^* \| < \varepsilon .
$$
Since $C$ is a hereditary subalgebra of $B$, we see that $b^{\frac{1}{4}}x_i \sigma(a^{\frac{1}{2}})\in C$ for any $i=1,2,...,L^{\prime}(a,\varepsilon)$. 
Therefore $\sigma$ is $(L^{\prime},N^{\prime})$-full in $C$. 
\end{proof}

Essentially the same proof as \cite[Lemma 8.15]{EN} show the following lemma. 
Roughly speaking, this lemma says that if target algebras have strict comparison, then the $(L,N)$-fullness can be controlled by traces. 

\begin{lem}\label{lem:EN-comparison} (Elliott-Niu) \ \\
For any $\varepsilon >0$ and $\delta >0$, there exist $\ell(\delta)\in\mathbb{N}$ and $n(\varepsilon)>0$ such that the following holds. 
Let $A$ be a C$^*$-algebra, and let $S$ be a subset of $T_1(A)$. Assume that $A$ has strict comparison with respect to $S$. 
If $a$ and $b$ are positive contractions in $A$ such that 
$$
\tau (a) > d_{\tau}(b)\delta
$$ 
for any $\tau\in S$, then there exist $x_1,x_2,...,x_{\ell(\delta)}\in A$ such that 
$$
\| x_i \| \leq n(\varepsilon)
$$ 
for any $i=1,2,...,\ell(\delta)$ and 
$$
\| b - \sum_{i=1}^{\ell(\delta)} x_i a x_{i}^* \| < \varepsilon .
$$
\end{lem}

The following lemma is a variant of \cite[Lemma 3.5]{M1}. 

\begin{lem}\label{lem:split}
Let $A$ be $\sigma$-unital C$^*$-algebra, and let $\tau$ be an extremal tracial state on $A$. If $a$ is a positive element in $F(A)$, then 
$$
\tau_{\omega} (\rho (a\otimes x)) = \tau_{\omega}(a) \tau (x)  
$$
for any $x\in A$. 
\end{lem}
\begin{proof}
There exists a positive contraction $(a_n)_n$ in $\mathcal{W}_{\omega}$ such that $a=[(a_n)_n]$. Note that we have $\rho (a\otimes x)= (a_nx)_n=(a_n^{\frac{1}{2}}xa_n^{\frac{1}{2}})_n$. 
For any $x\in A$, define $\tau^{\prime}(x) := \lim_{n\to\omega} \tau (a_nx)$. 
Then  $\tau^{\prime}$ is a trace on $A$ and $\tau^{\prime}\leq \tau$. Since $\tau$ is an extremal tracial state on $A$, there exists $\lambda \geq 0$ such that 
$\tau^{\prime} = \lambda \tau$. Let $\{h_m\}_{m\in\mathbb{N}}$ be an approximate unit for $A$. Then $\lim_{m\to\infty} \tau(h_m) = 1$ because $\tau$ is a state. 
Hence $\lambda = \lim_{m\to \infty}\tau^{\prime} (h_m)$. Similar arguments as in the proof of \cite[Proposition 5.3]{Na2} show 
$$
\lim_{m\to\infty} \lim_{n\to \omega} \tau (a_nh_m) = \lim_{n\to \omega} \tau (a_n) .
$$
Therefore $\lambda = \lim_{n\to \omega} \tau (a_n)$. We obtain the conclusion. 
\end{proof}

For a projection $p$ in $F( \mathcal{W})$, define a homomorphism $\sigma_p$ from $\mathcal{W}$ to $\mathcal{W}^\omega_p$ by 
$$
\sigma_p (x) := \rho (p\otimes x)
$$
for any $x\in \mathcal{W}$. 

\begin{pro}\label{pro:full-inclusion}
There exist maps $L: \mathcal{W}_{+,1}\setminus \{0\}\times (0,1)\to \mathbb{N}$ and 
$N: \mathcal{W}_{+,1}\setminus \{0\}\times (0,1) \to (0,\infty)$ such that the following holds. 
If $p$ be a projection in $F( \mathcal{W})$ such that $\tau_{\omega} (p)>0$ where $\tau$ is the unique tracial state on $\mathcal{W}$, 
then $\sigma_p$ is $(L,N)$-full. 
\end{pro}
\begin{proof}
For any $\varepsilon >0$ and $\delta >0$, take $\ell(\delta)\in\mathbb{N}$ and $n(\varepsilon)>0$ in Lemma \ref{lem:EN-comparison}. 
Define 
$$
L(a, \varepsilon ):=\ell \left(\frac{\tau (a)}{2} \right), \quad N(a, \varepsilon ):= n \left(\varepsilon\right) 
$$
for any $a\in \mathcal{W}_{+,1}\setminus \{0\}$ and $\varepsilon\in (0,1)$. 
Note that $\tau (a)>0$ since $\mathcal{W}$ is simple. 

Let $p\in F(\mathcal{W})$ be a projection in $F(\mathcal{W})$ such that $\tau_{\omega}(p) >0$. 
There exists a positive contraction $(p_n)_n$ in $\mathcal{W}_{\omega}$ such that $p=[(p_n)_n]$. 
Let $a\in \mathcal{W}_{+,1}\setminus\{0\}$, $b\in (\mathcal{W}^{\omega}_{p})_{+,1}$ and $\varepsilon >0$. 
Since $\mathcal{W}^{\omega}_{p}=\overline{\rho(p\otimes s)\mathcal{W}^{\omega}\rho (p\otimes s)}= \overline{(p_ns)_n\mathcal{W}^{\omega} (sp_n)_n}$ where $s$ is a strictly 
positive element in $\mathcal{W}$, we have $(p_n)_n b (p_n)_n=b$. Hence $d_{\tau_{\omega}} (b) \leq \tau_{\omega} (p)$. 
By Lemma \ref{lem:split} and $\tau_{\omega}(p)>0$, we have 
$$
\tau_{\omega} (\sigma_{p} (a) )= \tau_{\omega} (p) \tau (a) > \tau_{\omega} (p) \frac{\tau (a)}{2} \geq d_{\tau_{\omega}} (b) \frac{\tau (a)}{2}. 
$$
Since $\mathcal{W}$ has strict comparison and the tracial state on $\mathcal{W}$ is unique, $\mathcal{W}^{\omega}$ has strict comparison 
with respect to $\{\tau_{\omega}\}$ (see, for example, the proof of \cite[Lemma 1.23]{BBSTWW}). 
Therefore Lemma \ref{lem:EN-comparison} implies that
there exist 
elements $x_1,x_2,...,x_{L(a,\varepsilon)}$ in $\mathcal{W}^{\omega}$ such that 
$$
\| x_i \| \leq N(a,\varepsilon)
$$ 
for any $i=1,2,...,L(a,\varepsilon)$ and 
$$
\| b - \sum_{i=1}^{L(a,\varepsilon)} x_i \sigma_p(a) x_{i}^* \| < \varepsilon .
$$
Since $\mathcal{W}^{\omega}_{p}$ is a hereditary subalgebra of $\mathcal{W}^{\omega}$, 
we obtain the conclusion by Lemma \ref{lem:h-LN-full}.  
\end{proof}

The following corollary is an immediate consequence of Proposition \ref{pro:stable-uniqueness} and Proposition \ref{pro:full-inclusion}.

\begin{cor}\label{cor:stable-uniqueness} 
For any finite subsets $F_1\subset C(\Omega)$, $F_2\subset \mathcal{W}$, $\varepsilon>0$, 
there exist finite subsets $G_1\subset C(\Omega)$, $G_2\subset \mathcal{W}$, $m\in\mathbb{N}$  and $\delta >0$ such that the following holds. 
Let $p$ be a projection in $F(\mathcal{W})$ such that $\tau_{\omega} (p)>0$ where $\tau$ is the unique tracial state on $\mathcal{W}$. 
For any contractive ($G_1\odot G_2, \delta$)-multiplicative maps $\varphi, \psi : C(\Omega)\otimes \mathcal{W}\to \mathcal{W}^{\omega}_p$, 
there exist a unitary element $u$ in $M_{m^2+1}(\mathcal{W}^{\omega}_p)^{\sim}$ and $z_1,z_2,...,z_m\in\Omega$ such that 
\begin{align*}
\| u^{*}(\varphi(f\otimes a) \oplus & \overbrace{\bigoplus_{k=1}^m f(z_k)\rho (p\otimes a)\oplus \cdots \oplus\bigoplus_{k=1}^m f(z_k)\rho (p\otimes a) }^m) u \\
& - \psi(f\otimes a)\oplus \overbrace{\bigoplus_{k=1}^m f(z_k)\rho (p\otimes a) \oplus \cdots \oplus \bigoplus_{k=1}^m f(z_k)\rho (p\otimes a)}^m\| < \varepsilon 
\end{align*}
for any $f\in F_1$ and $a\in F_2$. 
\end{cor}

\section{Properties of $F(\mathcal{W})$}\label{sec:Properties}

In this section we shall consider properties of $F(\mathcal{W})$. In the rest of this paper, we denote by $\tau$ the unique tracial state on $\mathcal{W}$. 
Since $\mathcal{W}$ has property (SI), the following proposition is an immediate consequence of Theorem \ref{thm:Matui-Sato}. 

\begin{pro}\label{pro:w-comparison} 
(i) The central sequence C$^*$-algebra $F(\mathcal{W})$ has a unique tracial state $\tau_{\omega}$. \ \\
(ii) If $a$ and $b$ are positive elements in $F(\mathcal{W})$ satisfying $d_{\tau_{\omega}}(a)< d_{\tau_{\omega}} (b)$, then there exists an element $r\in F(\mathcal{W})$ such that 
$r^*br=a$. Moreover, $F(\mathcal{W})$ has strict comparison.
\end{pro}

The following proposition shows that $F(\mathcal{W})$ has many projections. 

\begin{pro}\label{pro:key-pro}
(i) For any $N\in\mathbb{N}$, there exists a unital homomorphism from $M_N(\mathbb{C})$ to $F(\mathcal{W})$. \ \\
(ii) For any $\theta\in [0,1]$, there exists a non-zero projection $p$ in $F(\mathcal{W})$ such that $\tau_{\omega}(p)=\theta$. \ \\
(iii) Let $h$ be a positive element in $F(\mathcal{W})$ such that $d_{\tau_{\omega}}(h)>0$. For any $\theta \in [0, d_{\tau_{\omega}}(h))$, 
there exists a non-zero projection $p$ in $\overline{hF(\mathcal{W})h}$ such that $\tau_{\omega}(p)=\theta$. 
\end{pro}

\begin{proof}
Let $\{h_n\}_{n\in\mathbb{N}}$ be an approximate unit for $\mathcal{W}$. 

(i) By Theorem \ref{thm:Razak}, $\mathcal{W}$ is isomorphic to $\mathcal{W}\otimes M_{N^{\infty}}= \mathcal{W}\otimes \bigotimes_{n\in\mathbb{N}} M_N(\mathbb{C})$. 
Define a map $\varphi$ from $M_N(\mathbb{C})$ to $F(\mathcal{W})\cong F(\mathcal{W}\otimes \bigotimes_{n\in\mathbb{N}} M_N(\mathbb{C}))$ by 
$$
\varphi (x) := [(h_n\otimes \overbrace{1\otimes \cdots \otimes 1}^n \otimes x \otimes 1\otimes \cdots )_n ] 
$$
for any $x\in M_{N}(\mathbb{C})$. 
Then $\varphi$ is a unital homomorphism. \ \\

(ii) 
Since $\mathbb{Z}[1/2]$ is dense in $\mathbb{R}$, there exists a sequence $\{q_n \}_{n\in\mathbb{N}}$ of non-zero projections in $M_{2^{\infty}}$ such that 
$\lim_{n\to \infty}\tau^{\prime}(q_n) =\theta$, where $\tau^{\prime}$ is the unique tracial state on $M_{2^{\infty}}$.  
Put 
$$
p:= [(h_n\otimes \overbrace{1\otimes \cdots \otimes 1}^n \otimes q_n \otimes 1\otimes \cdots)_n] \in F(\mathcal{W}\otimes \bigotimes_{n\in\mathbb{N}} M_{2^{\infty}})\cong F(\mathcal{W}).
$$
Then $p$ is a non-zero projection in $F(\mathcal{W})$ such that $\tau_{\omega}(p)=\theta$. \ \\

(iii) There exists a non-zero projection $q$ in $F(\mathcal{W})$ such that $\tau_{\omega}(q) = \theta$ by (ii). 
Proposition \ref{pro:w-comparison} implies that 
there exists an element $r$ in $F(\mathcal{W})$ such that $rhr^* = q$. Let $p:= h^{\frac{1}{2}}r^*rh^{\frac{1}{2}}$, then $p$ is a non-zero projection in $\overline{hF(\mathcal{W})h}$ such that 
$\tau_{\omega}(p)=\theta$. 
\end{proof}

Recall that $\mathrm{ker}(\varrho_{\mathcal{W}})= \{x\in F(\mathcal{W})\; |\; \tau_{\omega}(x^*x)=0\}$. 

\begin{pro}\label{pro:comparison-full}
Let $x$ be an element in $F(\mathcal{W})$. Then $x$ is full if and only if $x\notin \mathrm{ker}(\varrho_{\mathcal{W}})$. 
\end{pro}
\begin{proof}
It is obvious that if $x\in\mathrm{ker}(\varrho_{\mathcal{W}})$, then $x$ is not full in $F(\mathcal{W})$. 
Let $x\notin\mathrm{ker}(\varrho_{\mathcal{W}})$, then $\tau_{\omega} (x^*x) >0$. 
For any $b\in F(\mathcal{W})_{+}$, there exists a positive number $\delta$ such that $\tau_{\omega}(x^*x)>\delta d_{\tau_{\omega}}(b)$. 
Hence, for any $\varepsilon>0$, there exist  $z_1,z_2,...,z_{\ell(\delta)}\in F(\mathcal{W})$ such that 
$$
\| b - \sum_{i=1}^{\ell(\delta)}z_{i}x^*xz_{i}^*\| < \varepsilon 
$$
by Lemma \ref{lem:EN-comparison}. This shows the closed ideal generated by $x$ is equal to $F(\mathcal{W})$. 
Therefore $x$ is full. 
\end{proof}

Using the ideas in \cite{Ror1} and \cite{Rob2}, we shall show that certain elements in $F(\mathcal{W})$ can be approximated by invertible elements. 
We denote by $\mathrm{GL}(A)$ the set of invertible elements in $A$.

\begin{lem}\label{lem:ror1}
Let $a$ and $b$ be positive elements in $F(\mathcal{W})$ such that $a,b\notin\mathrm{ker}(\varrho_{\mathcal{W}}) $. Then there exist a unitary element $u$ and a projection 
$p^{\prime}$ in $F(\mathcal{W})$ such that $p^{\prime} \notin\mathrm{ker}(\varrho_{\mathcal{W}})$ and $p^{\prime}\in \overline{aF(\mathcal{W})a}\cap u(\overline{bF(\mathcal{W})b})u^*$. 
\end{lem}
\begin{proof}
Since $a,b\notin\mathrm{ker}(\varrho_{\mathcal{W}})$, we have $d_{\tau_{\omega}}(a)>0$ and $d_{\tau_{\omega}}(b)>0$.  Proposition \ref{pro:key-pro} and Proposition \ref{pro:w-comparison} imply that there exist a 
projection $p\in\overline{aF(\mathcal{W})a}$ and an element $r\in F(\mathcal{W})$ such that $p\notin\mathrm{ker}(\varrho_{\mathcal{W}})$ and $rbr^*=p$. Then we have $prb\notin\mathrm{ker}(\varrho_{\mathcal{W}})$, 
and hence $arb\notin \mathrm{ker}(\varrho_{\mathcal{W}})$. Since $r$ is a linear combination of four unitaries in $F(\mathcal{W})$, 
there exists a unitary element $w$ in $F(\mathcal{W})$ such that $awb\notin\mathrm{ker}(\varrho_{\mathcal{W}})$. Since $\mathrm{ker}(\varrho_{\mathcal{W}})$ is closed, 
essentially the same argument as in the proof of \cite[Lemma 3.4]{Ror1} show that there exist a unitary element $u$ and a positive element $c\in F(\mathcal{W})$ such that 
$c \notin\mathrm{ker}(\varrho_{\mathcal{W}})$ and $c \in \overline{aF(\mathcal{W})a}\cap u(\overline{bF(\mathcal{W})b})u^*$. 
Using Proposition \ref{pro:key-pro}, we can find a projection $p^{\prime}$ satisfying the conclusion of the proposition. 
\end{proof}

\begin{lem}\label{lem:reduce-nilpotent}
Let $x$ be an element in $F(\mathcal{W})$. Assume that there exist projections $p$ and $q$ in $F(\mathcal{W})$ such that 
$xp=qx=x$ and $1-p,1-q \notin \mathrm{ker}(\varrho_{\mathcal{W}})$. Then there exist a unitary element $u$ and a projection $e$ in $F(\mathcal{W})$ such that 
$e ux= ux e=ux$ and $1-e \notin\mathrm{ker}(\varrho_{\mathcal{W}})$. 
\end{lem}
\begin{proof}
Lemma \ref{lem:ror1} implies that there exist a unitary element $u$ and a projection $p^{\prime}$ in $F(\mathcal{W})$ such that $p^{\prime} \notin\mathrm{ker}(\varrho_{\mathcal{W}})$ 
and 
$$
p^{\prime}\in (1-p)F(\mathcal{W})(1-p)\cap u(1-q)F(\mathcal{W})(1-q)u^*.
$$ 
Since $x(1-p)=(1-q)x=0$, we have $uxp^{\prime}=0$ and $ p^{\prime} ux =0$. Put $e:=1-p^{\prime}$, then we obtain the conclusion. 
\end{proof}

\begin{lem}\label{lem:nilpotent}
Let $y$ be an element in $F(\mathcal{W})$. Assume that there exists a projection $e$ in $F(\mathcal{W})$ such that $ey=ye=y$ and $1-e\notin \mathrm{ker}(\varrho_{\mathcal{W}})$. 
Then $y\in \overline{\mathrm{GL}(F(\mathcal{W}))}$. 
\end{lem}
\begin{proof}
It is enough to show that $y$ is a product of two nilpotent elements. Since $1-e$ is full by Proposition \ref{pro:comparison-full}, there exist 
elements $x_1,x_2,...,x_{N}$ in $F(\mathcal{W})$ such that $\sum_{i=1}^N x_i (1-e)x_i^* =1$. 
Proposition \ref{pro:key-pro} implies that there exists a unital homomorphism $\varphi$ from $M_{N+1}(\mathbb{C})$ to $F(\mathcal{W})$. 
Let $\{e_{ij}\}_{i,j=1}^{N+1}$ be the standard matrix units of $M_{N+1}(\mathbb{C})$. 
Taking suitable subsequences of representatives of $\varphi (e_{ij})$ for any $i,j=1,2,...,N+1$, we may assume that the range of $\varphi$ commutes with $y$, $e$, 
$x_1$,...,$x_N$. 
Put 
$$
r_1:= \sum_{i=1}^N ex_i(1-e) \varphi (e_{N+1i}), \quad r_2:=\sum_{i=1}^{N} y \varphi (e_{ii+1}),  \quad r:=r_1+r_2  
$$
and 
$$
t_1:=\sum_{i=1}^N (1-e)x_i^* y \varphi (e_{iN+1}), \quad t_2:= \sum_{i=1}^N e\varphi (e_{i+1i}), \quad t:=t_1+t_2.
$$
Then similar arguments as in the proof of \cite[Lemma 2.1]{Rob2} show that $rt=y$ and $r^{N+2}=t^{N+2}=0$. 
Indeed, we have 
$$
rt= (r_1+r_2)(t_1+t_2)= r_1t_1+r_2t_2= \sum_{i=1}^N ex_i(1-e)x_i^*y \varphi (e_{N+1N+1})+ \sum_{i=1}^N y\varphi (e_{ii})=y. 
$$
Since we have $r_1r_2=0$, $r^k=(r_1+r_2)^k=\sum_{i=0}^k r_2^ir_1^{k-i}$ for any $k\in\mathbb{N}$. It can be easily checked that 
$r_1^2=0$ and $r_2^{N+1}=0$. This implies that $r^{N+2}=0$. In a similar way, we see that $t^{N+2}=0$. 
\end{proof}

The following proposition is an immediate consequence of Lemma \ref{lem:reduce-nilpotent} and Lemma \ref{lem:nilpotent}.

\begin{pro}\label{pro:almost-one}
Let $x$ be an element in $F(\mathcal{W})$. Assume that there exist projections $p$ and $q$ in $F(\mathcal{W})$ such that 
$xp=qx=x$ and $1-p,1-q \notin \mathrm{ker}(\varrho_{\mathcal{W}})$. Then $x\in \overline{\mathrm{GL}(F(\mathcal{W}))}$. 
\end{pro}

Using the proposition above, 
we shall show that for certain projections in $F(\mathcal{W})$, Murray-von Neumann equivalence and unitary equivalence coincide. 

\begin{pro}\label{pro:MvN-u}
Let $p$ and $q$ be projections in $F(\mathcal{W})$ such that $\tau_{\omega} (p)<1$. 
Then $p$ and $q$ are Murray-von Neumann equivalent if and only if $p$ and $q$ are unitarily equivalent. 
\end{pro}
\begin{proof}
The if part is obvious. We will show the only if part. 
Suppose that there exists a partial isometry $v\in F(\mathcal{W})$ such that $v^*v=p$ and $vv^*=q$. 
Since we have $vp=qv=v$ and $\tau_{\omega}(1-q)=\tau_{\omega}(1-p)>0$, there exists an invertible element $s$ of norm one such that 
$\| s-v \| <1/4$ by Proposition \ref{pro:almost-one}. 
Let $u:= s(s^*s)^{-\frac{1}{2}}$. Then $u$ is a unitary element in $F(\mathcal{W})$ and we have 
$\| up u^* - q \| <1$. Therefore we see that $p$ is unitarily equivalent to $q$. 
\end{proof}

We shall show that every unitary element in $F(\mathcal{W})$ can be lifted to a unitary element in $\mathcal{W}^{\sim}_{\omega}$. 

\begin{pro}\label{pro:u-lift}
Let $A$ be a $\sigma$-unital C$^*$-algebra with $A\subset \overline{\mathrm{GL}(A^{\sim})}$.  If $u$ is a unitary element in $F(A)$, 
then there exists a unitary element $w$ in $A^{\sim}_{\omega}$ such that $u= [w]$. 
\end{pro}
\begin{proof}
Since $A\subset \overline{\mathrm{GL}(A^{\sim})}$, there exists a bounded sequence $\{z_n \}_{n\in\mathbb{N}}$ of invertible elements in $A^{\sim}$ 
such that $u= [(z_n)_n]$. Note that for any $a\in A$,  
$$
\lim_{n\to\omega} \|z_n^*z_na -a \| =0, \quad \lim_{n\to\omega} \|z_nz_n^*a -a \| =0
$$
because $u$ is a unitary element in $F(A)$. For any $n\in\mathbb{N}$, let $w_n: = z_n(z_n^*z_n)^{-\frac{1}{2}}$. Then 
$w_n$ is a unitary element in $A^{\sim}$ and for any $a\in A$, we have 
$$
\| w_na-z_na \| = \| w_n( a- (z_n^*z_n)^{\frac{1}{2}}a) \| = \| a- (z_n^*z_n)^{\frac{1}{2}}a \| \to 0
$$
as $n\to \omega$. Furthermore, for any $a\in A$, we have 
\begin{align*}
\| [w_n ,a] \| 
& = \| w_naw_n^* -a \| \\
& = \| w_naw_n^* -z_na w_n^*  +  z_na w_n^*-z_naz_n^*  +  z_naz_n^* - z_nz_n^*a  +  z_nz_n^*a -a \| \\
& \leq \| w_na -z_na  \| + \| z_n \| \|  w_na^*-z_na^* \| + \| z_n\| \|az_n^* - z_n^*a \| + \| z_nz_n^*a -a \| \\
& \to 0
\end{align*}
as $n\to \omega$. Therefore $(w_n)_n$ is a unitary element in $A^{\sim}_{\omega}$ such that $u= [(w_n)_n]$. 
\end{proof}

Since $\mathcal{W}$ has stable rank one, we have the following corollary.  

\begin{cor}\label{cor:u-lift}
Let $u$ be a unitary element in $F(\mathcal{W})$. Then there exists a unitary element $w$ in $\mathcal{W}^{\sim}_{\omega}$ such that $u= [w]$. 
\end{cor}

\section{Homotopy of unitaries in $F(\mathcal{W})$}\label{sec:Homotopy}

In this section we shall prove Theorem \ref{thm:homotopy}. The following lemma is motivated by \cite[Lemma 4.1]{M2} and \cite[Lemma 4.2]{M2}. 

\begin{lem}\label{lem:m4.2}
Let $\Omega$ be a compact metrizable space, and let $F$ be a finite subset of $C(\Omega)$ and $\varepsilon >0$. 
Suppose that $\varphi$ and $\psi$ are unital homomorphisms from $C(\Omega)$ to $F(\mathcal{W})$ 
such that 
$
\tau_{\omega} \circ \varphi  = \tau_{\omega} \circ \psi .
$ 
Then there exist a projection $p\in F(\mathcal{W})$, $(F,\varepsilon)$-multiplicative unital c.p. maps 
$\varphi^{\prime}$ and $\psi^{\prime}$ from $C(\Omega)$ to $pF(\mathcal{W})p$ and a unital homomorphism $\sigma$ from $C(\Omega)$ to $(1-p)F(\mathcal{W})(1-p)$ with 
finite-dimensional range such that 
$$
0 <\tau_{\omega} (p) < \varepsilon, \quad \varphi \sim_{F,\varepsilon} \varphi^{\prime}\oplus \sigma , \quad \psi \sim_{F,\varepsilon} \psi^{\prime}\oplus \sigma .
$$  
\end{lem}
\begin{proof}
We may assume that every element in $F$ is of norm one. 
Let $\mu$ be the probability measure on $\Omega$ corresponding to $\tau_{\omega}\circ \varphi=\tau_{\omega}\circ \psi$. By the same argument as in the proof of
\cite[Lemma 4.1]{M2}, there exist pairwise disjoint open subsets $W_1,W_2,...,W_l\subset \Omega$ such that
$$
\mu (\Omega\setminus \bigcup_{i=1}^l W_i)=0, \quad \mu (W_i)>0
$$
and $|f(x)-f(y)|<\varepsilon /3$ for any $x,y\in W_i$ and $f\in F$. For any $i=1,2,...,l$, choose $z_i\in W_{i}$. 
Proposition \ref{pro:key-pro} implies that there exists a projection $p_i$ in 
$
\overline{\varphi ( C_0(W_{i}))F(\mathcal{W})\varphi (C_0(W_{i}))}
$
such that 
$$
 \mu (W_i) -\frac{\varepsilon}{l} < \tau_{\omega}(p_i) < \mu (W_i) .
$$ 
Note that we have 
$$
\| p_i\varphi (f) - f(z_i)p_i \| < \frac{\varepsilon}{3}, \quad \| \varphi (f)p_i - f(z_i)p_i \| < \frac{\varepsilon}{3}
$$ 
for any $f\in F$ and $i=1,2,..., l$. In the same way as in 
the proof of \cite[Lemma 4.1]{M2}, we see that there exist mutually orthogonal projections $\overline{p}_1,\overline{p}_2,...,\overline{p}_l$ in $F(\mathcal{W})^{**}$ such that 
$\overline{p}_i$ commutes with $\varphi (C(\Omega))$ and $p_i\leq \overline{p}_i$. 
In a similar way as above, there exist a projection $q_i$ in 
$
\overline{\psi ( C_0(W_{i}))F(\mathcal{W})\psi (C_0(W_{i}))}
$
such that 
$$
\tau_{\omega}(p_i)  < \tau_{\omega}(q_i) < \mu (W_i)
$$ 
and 
$$
\| q_i\psi (f) - f(z_i)q_i \| < \frac{\varepsilon}{3}, \quad \| \psi (f)q_i - f(z_i)q_i \| < \frac{\varepsilon}{3}
$$ for any $f\in F$ and $i=1,2..., l$. Also, there exist mutually orthogonal  projections $\overline{q}_1,\overline{q}_2,...,\overline{q}_l$ in 
$F(\mathcal{W})^{**}$ such that $\overline{q}_i$ commutes with $\psi (C(\Omega))$ and $q_i\leq \overline{q}_i$.

For any $i=1,2,...,l$, there exists a subprojection $q_i^{\prime}$ of $q_i$ such that $q_i^{\prime}$ is Murray-von Neumann equivalent to $p_i$ by Proposition \ref{pro:w-comparison}. 
It follows from Proposition \ref{pro:MvN-u} that there exists a unitary element $u$ in $F(\mathcal{W})$ such that $u p_iu^* =q_i^{\prime}$ for any $i=1,2,...,l$. 
Put $p:= 1-\sum_{i=1}^l p_i$ and $q:=1-\sum_{i=1}^l q_i^{\prime}$. Then we have 
$$
0 < \tau_{\omega} (p) < \varepsilon .
$$
Since $(1-p)\varphi (f)= \sum_{i=1}^l p_i\varphi (f)= \sum_{i=1}^l p_i\overline{p}_i \varphi (f) = \sum_{i=1}^l p_i \varphi (f) \overline{p}_i$, we have 
$$
\| (1-p)\varphi (f) - \sum_{i=1}^l f(z_i) p_i \| < \frac{\varepsilon}{3}
$$
for any $f\in F$. 
Moreover, we have 
$$
\| [p, \varphi (f) ] \| < \frac{2\varepsilon}{3}
$$
for any $f\in F$. In a similar way, we have 
$$
\| (1-q)\psi (f) - \sum_{i=1}^l f(z_i) q_i^{\prime} \| < \frac{\varepsilon}{3}, \quad \| [q, \psi (f) ] \| < \frac{2\varepsilon}{3} 
$$
for any $f\in F$. 

Define unital c.p. maps $\varphi^{\prime}$ and $\psi^{\prime}$ from $C(\Omega)$ to $pF(\mathcal{W})p$ by 
$$
\varphi^{\prime} (f):= p\varphi (f) p, \quad \psi^{\prime} (f) := p u^* \psi (f) u p, 
$$
and define a unital homomorphism $\sigma$ from $C(\Omega)$ to $(1-p)F(\mathcal{W})(1-p)$ by 
$$
\sigma (f) : = \sum_{i=1}^l f(z_i)p_i. 
$$
Then it is easy to see that $\varphi^{\prime}$ and $\psi^{\prime}$ are $(F,2\varepsilon /3)$-multiplicative maps. 
We have 
\begin{align*}
\| \varphi (f) - (\varphi^{\prime}(f) + \sigma (f)) \| 
& \leq \| p\varphi (f)- p\varphi (f)p \| + \| (1-p) \varphi (f) - \sum_{i=1}^l f(z_i)p_i \| \\
& < \frac{2\varepsilon}{3} + \frac{\varepsilon}{3} = \varepsilon 
\end{align*}
for any $f\in F$. 
Also, we have 
$$
\| \psi (f) - u(\psi^{\prime}(f) + \sigma (f))u^* \| = \| \psi(f) - q\psi(f)q- \sum_{i=1}^l f(z_i)q_i^{\prime} \| < \varepsilon  
$$
for any $f\in F$. Therefore the proof is complete. 
\end{proof}

The following theorem is related to \cite[Theorem 4.5]{M2}.

\begin{thm}\label{thm:unitary-equivalence-ed}
Let $\Omega$ be a compact metrizable space, and let $F_1$ be a finite subset of $C(\Omega)$ and $F_2$ a finite subset of $\mathcal{W}$, and let $\varepsilon >0$. 
Then there exist mutually orthogonal positive elements $h_1,h_2,...,h_{l}$ in $C(\Omega)$ of norm one such that the following holds. 
For any $\nu >0$, there exist finite subsets $G_1\subset C(\Omega)$, $G_2\subset \mathcal{W}$ and $\delta >0$ such that the following holds. 
If $\varphi$ and $\psi$ are  unital c.p. maps from $C(\Omega)$ to $M(\mathcal{W})$ such that  
$$
\tau (\varphi (h_i)) \geq \nu, \quad 1 \leq \forall i \leq l, 
$$
$$
\| [\varphi (f), a] \| < \delta , \quad \| [\psi(f),a ] \| < \delta,  \quad \forall f\in G_1, a\in G_2,
$$
$$
\| (\varphi (fg)- \varphi (f)\varphi (g))a\| < \delta, \quad \| (\psi (fg)- \psi (f)\psi (g))a\| < \delta, \quad \forall f,g\in G_1, a\in G_2,
$$
$$
| \tau (\varphi (f)) -\tau (\psi (f)) | < \delta, \quad \forall f\in G_1,
$$
then there exists a unitary element $u$ in $\mathcal{W}^{\sim}$ such that 
$$
\| u\varphi (f)au^* - \psi(f)a \| < \varepsilon 
$$
for any $f\in F_1$ and $a\in F_2$. 
\end{thm}

\begin{proof}
We may assume that every element in $F_2$ is of norm one. 
Let $\{y_1,y_2,...,y_l \}$ be a finite subset of $\Omega$ such that for any $x\in \Omega$, there exists $y_i\in\{y_1,y_2,...,y_l \}$ such that 
$| f(x)-f(y_i)  |< \varepsilon /7$ for any $f\in F_1$. Choose pairwise disjoint open neighborhoods $W_{1},W_{2},...,W_{l}$ of $y_1,y_2,...,y_l$ respectively such that 
if $x\in W_{i}$ and $f\in F_1$, then $| f(x)-f(y_i) | < \varepsilon /7$. 
For any $i=1,2,...,l$, take a positive element $h_i\in C_0(W_i)$ of norm one. 
We shall show that $h_1,h_2,...,h_l$ have the desired property. On the contrary, suppose that $h_1,h_2,...,h_l$ did not have the desired property. 
Then there exists a positive number $\nu$ satisfying the following: For any $n\in\mathbb{N}$, there exist unital c.p. maps 
$\varphi_n, \psi_n : C(\Omega)\to M(\mathcal{W})$ such that 
$$
\tau (\varphi_n (h_i)) \geq \nu, \quad 1\leq \forall i \leq l,
$$
$$
\| [\varphi_n(f),a]\| \to 0, \quad \| [\psi_n(f),a]\| \to 0, \quad \| (\varphi_n(fg)- \varphi_n(f)\varphi_n(g))a\| \to 0,  
$$
$$
\| (\psi_n(fg)- \psi_n (f)\psi_n(g))a\| \to 0,\quad |\tau (\varphi_n(f))-\tau (\psi_n(f))| \to 0
$$
as $n\to\infty$ for any $f,g\in C(\Omega)$, $a\in \mathcal{W}$ and 
$$
\max_{f\in F_1, a\in F_2} \| u\varphi_n (f)au^* - \psi_n (f)a \| \geq \varepsilon 
$$
for any unitary element $u$ in $\mathcal{W}^{\sim}$. 

Define homomorphisms $\varphi$ and $\psi$ from $C(\Omega)$ to $F(\mathcal{W})$ by 
$$
\varphi (f): = [(\varphi_n(f))_n ], \quad \psi(f):= [(\psi_n(f))_n]
$$
for any $f\in C(\Omega)$, and define homomorphisms $\Phi$ and $\Psi$ from $C(\Omega)\otimes \mathcal{W}$ to $\mathcal{W}^{\omega}$ by
$$
\Phi := \rho \circ (\varphi \otimes \mathrm{id}_{\mathcal{W}}), \quad 
\Psi := \rho \circ (\psi \otimes \mathrm{id}_{\mathcal{W}}).
$$
Note that we have 
$$
\tau_{\omega}(\varphi (h_i))\geq \nu
$$ for any $i=1,2,...,l$ and 
$$
\tau_{\omega}\circ \varphi=\tau_{\omega} \circ \psi .
$$
 
Applying Corollary \ref{cor:stable-uniqueness} to $F_1$, $F_2$ and $\varepsilon /7 $, we obtain finite subsets $G_1 \subset C(\Omega)$ $G_2 \subset \mathcal{W}$, $m\in\mathbb{N}$ and 
$\delta>0$. Put 
$$
F_1^{\prime}:= F_1\cup G_1 \cup \{h_1, h_2,...,h_l \}, \quad
\varepsilon^{\prime}:= \min\{\varepsilon /7, \delta, \nu /(m^2+2) \}.
$$
Applying Lemma \ref{lem:m4.2} to $F_1^{\prime}$, $\varepsilon^{\prime}$, $\varphi$ and $\psi$, there exist a projection $p\in F(\mathcal{W})$, $(F^{\prime},\varepsilon^{\prime})$-multiplicative unital c.p. maps 
$\varphi^{\prime}$ and $\psi^{\prime}$ from $C(\Omega)$ to $pF(\mathcal{W})p$ and a unital homomorphism $\sigma$ from $C(\Omega)$ to $(1-p)F(\mathcal{W})(1-p)$ with 
finite-dimensional range such that 
$$
0<\tau_{\omega} (p) < \varepsilon^{\prime}, \quad \varphi \sim_{F^{\prime}_1,\varepsilon^{\prime}} \varphi^{\prime}\oplus \sigma , \quad \psi \sim_{F^{\prime}_1,\varepsilon^{\prime}} \psi^{\prime}\oplus \sigma .
$$  
Define c.p. maps $\Phi^{\prime}$ and $\Psi^{\prime}$ from $C(\Omega)\otimes \mathcal{W}$ to $\mathcal{W}^{\omega}_p$ by 
$$
\Phi^{\prime} := \rho \circ (\varphi^{\prime}\otimes \mathrm{id}_{\mathcal{W}}), \quad 
\Psi^{\prime} := \rho \circ (\psi^{\prime}\otimes \mathrm{id}_{\mathcal{W}})
$$
and define a homomorphism $\Sigma$ from $C(\Omega)\otimes \mathcal{W}$ to $\mathcal{W}^{\omega}_{1-p}$ by 
$$
\Sigma := \rho \circ (\sigma\otimes\mathrm{id}_{\mathcal{W}}).
$$
Since every unitary element in $F(\mathcal{W})$ can be lifted to a unitary element in $\mathcal{W}_{\omega}^{\sim}$ by Corollary \ref{cor:u-lift}, 
$$
\Phi \sim_{F_1\odot F_2, \frac{\varepsilon}{7}} \Phi^{\prime}\oplus \Sigma , \quad \Psi \sim_{F_1\odot F_2, \frac{\varepsilon}{7}} \Psi^{\prime}\oplus \Sigma .
$$
It can be easily checked that $\Phi^{\prime}$ and $\Psi^{\prime}$ are contractive $(G_1 \odot G_2 , \delta)$-multiplicative maps. 
By Corollary \ref{cor:stable-uniqueness}, there exist a unitary element $U$ in $M_{m^2+1}(\mathcal{W}^{\omega}_p)^{\sim}$ and $z_1,z_2,...,z_m\in\Omega$ such that 
\begin{align*}
\| U^{*}(\Phi^{\prime} (f\otimes a) \oplus & \overbrace{\bigoplus_{k=1}^m f(z_k)\rho (p\otimes a)\oplus \cdots \oplus\bigoplus_{k=1}^m f(z_k)\rho (p\otimes a) }^m) U \\
& - \Psi^{\prime}(f\otimes a)\oplus \overbrace{\bigoplus_{k=1}^m f(z_k)\rho (p\otimes a) \oplus \cdots \oplus \bigoplus_{k=1}^m f(z_k)\rho (p\otimes a)}^m\| < \frac{\varepsilon}{7} 
\end{align*}
for any $f\in F_1$ and $a\in F_2$. 

For any homomorphism $\gamma : C(\Omega)\to F(\mathcal{W})$, let $\mu_{\gamma}$ denote the probability measure on $\Omega$ corresponding to $\tau_{\omega}\circ \gamma$. 
For any $i=1,2,...,l$, we have 
$$
\mu_{\sigma} (W_i) \geq \tau_{\omega}(\sigma (h_i)) > \tau_{\omega} (\varphi (h_i)) - \tau_{\omega} (\varphi^{\prime} (h_i)) - \varepsilon^{\prime} \geq \nu -\tau_{\omega}(p) -\varepsilon^{\prime} 
> \nu -2\varepsilon^{\prime} \geq m^{2}\varepsilon^{\prime}. 
$$
Hence we see that there exists a homomorphism $\sigma^{\prime}: C(\Omega)\to (1-p)F(\mathcal{W})(1-p)$ with finite-dimensional range such that 
$$
\| \sigma (f) - \sigma^{\prime} (f) \| < \frac{\varepsilon}{7}
$$
for any $f\in F_1$ and $\mu_{\sigma^{\prime}} (\{ y_i\})> m^2\varepsilon^{\prime}$ for any $i=1,2,...,l$ because of the property of $W_{i}$. 
Using Proposition \ref{pro:w-comparison}, we see that there exist mutually orthogonal projections $\{p_{j,k} \}_{j,k=1}^m$ in $(1-p)F(\mathcal{W})(1-p)$ 
and a homomorphism $\sigma^{\prime\prime}: C(\Omega)\to (1-p-q)F(\mathcal{W})(1-p-q)$ where $q=\sum_{j,k=1}^mp_{j,k}$ such that 
$$
\| \sigma^{\prime} (f) - (\sum_{j=1}^m\sum_{k=1}^mf(z_k)p_{j,k} + \sigma^{\prime\prime} (f) ) \| < \frac{\varepsilon}{7}
$$
for any $f\in F_1$ and $p_{j,k}$ is Murray-von Neumann equivalent to $p$ for any $j,k=1,2,...,m$ because of the property of $\{y_1,y_2,...,y_l\}$. 

Since 
$
\Phi^{\prime} (f\otimes a) + \sum_{j=1}^m\sum_{k=1}^m f(z_k)\rho (p_{j,k}\otimes a)
$ 
in $\mathcal{W}_{p+q}^{\omega}$ can be regarded as an element in $M_{m^2+1}(\mathcal{W}^{\omega}_p)^{\sim}$, 
there exists a unitary element $\widehat{U}$ in $(\mathcal{W}_{p+q}^{\omega})^{\sim}$ such that 
\begin{align*}
\| \widehat{U}^{*}(\Phi^{\prime}  (f\otimes a) + \sum_{j=1}^m\sum_{k=1}^m & f(z_k)\rho (p_{j,k}\otimes a) ) \widehat{U} \\
& -(\Psi^{\prime}(f\otimes a)+ \sum_{j=1}^m\sum_{k=1}^mf(z_k)\rho (p_{j,k}\otimes a)) \| < \frac{\varepsilon}{7}
\end{align*}
by the argument above. Note that we may assume that $\widehat{U}= (a_n)_n+1$ for some $(a_n)_n\in \mathcal{W}^{\omega}_{p+q}$. 
Let $V$ be a trivial unitary extension in $(\mathcal{W}^{\omega})^{\sim}$ from $\widehat{U}$. 
Then we have 
\begin{align*}
\| V^*(\Phi^{\prime}(f\otimes a)+ & \sum_{j=1}^m\sum_{k=1}^mf(z_k)\rho (p_{j,k}\otimes a) + \sigma^{\prime\prime} (f) )V \\
& -(\Psi^{\prime}(f\otimes a)+\sum_{j=1}^m\sum_{k=1}^mf(z_k)\rho (p_{j,k}\otimes a) + \sigma^{\prime\prime} (f) ) \|< \frac{\varepsilon}{7}.
\end{align*}
Let $\Sigma^{\prime}$ and $\Sigma^{\prime\prime}$ be homomorphisms from $C(\Omega)\otimes \mathcal{W}$ to $\mathcal{W}^{\omega}_{1-p}$ such that 
$$
\Sigma^{\prime}(f\otimes a)= \rho (\sigma^{\prime} (f) \otimes a), \quad \Sigma^{\prime\prime}(f\otimes a)= \rho \left(\sum_{j=1}^m\sum_{k=1}^m (f(z_k)p_{j,k} + \sigma^{\prime\prime} (f))\otimes a \right) 
$$
for any $f\in C(\Omega)$ and $a\in\mathcal{W}$. Then we have 
\begin{align*}
\Phi 
& \sim_{F_1\odot F_2, \frac{\varepsilon}{7}} \Phi^{\prime}\oplus \Sigma \sim_{F_1\odot F_2, \frac{\varepsilon}{7}} \Phi^{\prime}\oplus \Sigma^{\prime} \sim_{F_1\odot F_2, \frac{\varepsilon}{7}}
\Phi^{\prime}\oplus \Sigma^{\prime\prime} \sim_{F_1\odot F_2, \frac{\varepsilon}{7}} \Psi^{\prime}\oplus \Sigma^{\prime\prime} \\
& \sim_{F_1\odot F_2, \frac{\varepsilon}{7}} \Psi^{\prime}\oplus \Sigma^{\prime} 
\sim_{F_1\odot F_2, \frac{\varepsilon}{7}} \Psi^{\prime}\oplus \Sigma \sim_{F_1\odot F_2, \frac{\varepsilon}{7}} \Psi .
\end{align*}
Therefore there exists a unitary element $(w_n)_n$ in $(\mathcal{W}^{\omega})^{\sim}$ such that 
$$
\| (w_n)_n \Phi (f\otimes a)(w_n)_n^* - \Psi (f\otimes a) \| < \varepsilon  
$$
for any $f\in F_1$ and $a\in F_2$. Note that we may assume that $w_n$ is a unitary element in $\mathcal{W}^{\sim}$ for any $n\in\mathbb{N}$. 
Taking a sufficiently large $n$, we obtain a contradiction. Consequently, the proof is complete. 
\end{proof}

Let $\mathbb{T}$ be the unit circle in the complex plane. We denote by $\iota$ the identity function on $\mathbb{T}$, 
that is, $\iota (z)=z$ for any $z\in\mathbb{T}$.

\begin{thm}\label{thm:unitary-equivalence}
Let $u$ and $v$ be unitaries in $F(\mathcal{W})$ such that $\tau_{\omega} (f(u)) >0$ for any $f\in C(\mathbb{T})_{+}\setminus \{0\}$. 
Then there exists a unitary element $w$ in $F(\mathcal{W})$ such that $wuw^* =v$ if and only if 
$
\tau_{\omega} (f(u))= \tau_{\omega} (f(v))
$ 
for any $f\in C(\mathbb{T})$. 
\end{thm}
\begin{proof}
The only if part is obvious. We will show the if part. 
By Corollary \ref{cor:u-lift}, there exist unitaries $(u_n)_n$ and $(v_n)_n$ in $\mathcal{W}^{\sim}_{\omega}$ such that 
$u=[(u_n)_n]$ and $v=[(v_n)_n]$. 
For any $n\in\mathbb{N}$, define unital homomorphisms $\varphi_n$ and $\psi_n$ from $C(\mathbb{T})$ to $\mathcal{W}^{\sim}$ by $\varphi_n(f):= f(u_n)$ and $\psi_n(f):= f(v_n)$, respectively. 
Then we have
$$
| \tau (\varphi_n (f)) - \tau_{\omega} (f(u)) | \to 0, \quad \| [\varphi_n (f), a ] \| \to 0, \quad \| [\psi_n (f), a] \| \to 0,
$$
$$
| \tau (\varphi_n (f)) - \tau (\psi_n (f) ) | \to 0 
$$
as $n\to \omega$ for any $f\in C(\mathbb{T})$ and $a\in \mathcal{W}$. 

Let $F_1:=\{1, \iota\} \subset C(\mathbb{T})$, and let $\{F_{2,k}\}_{k\in\mathbb{N}}$ be a sequence of finite subsets of $\mathcal{W}$ such that 
$F_{2,k}\subset F_{2,k+1}$ and $\mathcal{W}=\overline{\bigcup_{k\in\mathbb{N}} F_{2,k}}$. 
For any $k\in\mathbb{N}$, applying Theorem \ref{thm:unitary-equivalence-ed} to $F_1$, $F_{2,k}$ and $1/k$, we obtain 
mutually orthogonal positive elements $h_{1,k},h_{2,k},...,h_{l(k),k}$ in $C(\mathbb{T})$ of norm one. 
Let 
$$
\nu_k := \frac{1}{2} \min\{\tau_{\omega} (h_{1,k}(u)),\tau_{\omega} (h_{2,k}(u)),...,\tau_{\omega} (h_{l(k),k}(u)) \} >0. 
$$ 
Applying Theorem \ref{thm:unitary-equivalence-ed} to $\nu_k$, we obtain finite subsets $G_{1,k}\subset C(\mathbb{T})$, $G_{2,k}\subset \mathcal{W}$ and 
$\delta_k>0$. We may assume that $G_{1,k}\subset G_{1,k+1}$, $G_{2,k}\subset G_{2,k+1}$ and $\delta_{k}> \delta_{k+1}$. 
It can be easily checked that there exists a sequence $\{X_k\}_{k\in\mathbb{N}}$ of elements in $\omega$ such that $X_k\supset X_{k+1}$ and 
for any $n\in X_{k}$, 
$$
| \tau (\varphi_n (h_{i,k})) - \tau_{\omega} (h_{i,k}(u)) | < \frac{\tau_{\omega} (h_{i,k}(u))}{2}, \quad 1\leq \forall i \leq l(k),
$$
$$
\| [\varphi_n (f), a ] \| < \delta_{k}, \quad \| [\psi_n (f), a] \|< \delta_k, \quad \forall f\in G_{1,k}, a\in G_{2,k},
$$
$$
| \tau (\varphi_n (f)) - \tau (\psi_n (f) ) | < \delta_k, \quad \forall f\in G_{1,k}.
$$
Since we have $\tau (\varphi_n (h_{i,k})) \geq \nu_k$ by the above, Theorem \ref{thm:unitary-equivalence-ed} implies that for any $n\in X_{k}$, there exists a unitary element 
$w_{k,n}$ in $\mathcal{W}^{\sim}$ such that 
$$
\| w_{k,n} \varphi_n(f)a w_{k,n}^* - \psi_n(f)a \| < \frac{1}{k}
$$
for any $f\in F_1$ and $a\in F_{2,k}$. Since $F_1=\{1, \iota\}$, we have 
$$
\| [w_{k,n}, a ] \| < \frac{1}{k}, \qquad \| w_{k,n} u_n a w_{k,n}^* - v_n a \| < \frac{1}{k}
$$
for any $n\in X_{k}$ and $a\in F_{2,k}$. Let 
$$
w_{n} := \left\{\begin{array}{cl}
1 & \text{if } n\notin X_1   \\
w_{k,n} & \text{if } n\in X_k\setminus X_{k+1}\quad (k\in\mathbb{N})
\end{array}
\right..
$$
Then we have 
$$
\| [w_n ,a ] \| \to 0, \quad \| w_n u_n w_n^*a - v_na \| \to 0
$$
as $n\to \omega$ for any $a\in \mathcal{W}$. Therefore $[(w_n)_n]$ is a unitary element in $F(\mathcal{W})$ and 
$$
[(w_n)_n]u [(w_n)_n]^* = v. 
$$
\end{proof}

Hiroki Matui told us the following lemma. 

\begin{lem}\label{lem:unitary-uhf}
For any  faithful tracial state $\tau_0$ on $C(\mathbb{T})$, there exists a unital homomorphism $\varphi$ from $C(\mathbb{T})$ to $M_{2^{\infty}}$ such that 
$\tau_0 = \tau^{\prime} \circ \varphi$ where $\tau^{\prime}$ is the unique tracial state on $M_{2^{\infty}}$.  
\end{lem}
\begin{proof}
We identify $C(\mathbb{T})$ with $\{f\in C([0,1])\; |\; f(0)=f(1) \}$. Note that $\tau_{0}$ extends to a faithful tracial state $\tilde{\tau}_{0}$ on $C([0,1])$. 
By \cite[Theorem 2.1 (i)]{Ror}, there exists a unital homomorphism $\psi$ from $C([0,1])$ to $\mathcal{Z}$ such that $\tilde{\tau}_0=\tau_{\mathcal{Z}}\circ \psi$, 
where $\tau_{\mathcal{Z}}$ is the unique tracial state on $\mathcal{Z}$. Define a unital homomorphism $\varphi$ from $C(\mathbb{T})$ to $M_{2^{\infty}}\otimes \mathcal{Z}$ by 
$\varphi := 1\otimes \psi|_{C(\mathbb{T})}$. Since $M_{2^{\infty}}$ is $\mathcal{Z}$-stable, we obtain the conclusion. 
\end{proof}

Note that we identify $F(\mathcal{W})$ with $F(\mathcal{W}\otimes M_{2^{\infty}})$ in the following lemmas.

\begin{lem}\label{lem:induce-uhf-unitary}
Let $u$  be a unitary element in $F(\mathcal{W})$ such that $\tau_{\omega} (f(u)) >0$ for any $f\in C(\mathbb{T})_{+}\setminus \{0\}$. Then there exist 
a unitary element $(v_n)_n$ in $(M_{2^{\infty}})_{\omega}$ and a unitary element $w$ in $F(\mathcal{W})$ such that 
$$
wuw^* = [(h_n\otimes v_n)_n]
$$
where $\{h_{n}\}_{n\in\mathbb{N}}$ is an approximate unit for $\mathcal{W}$. 
\end{lem}
\begin{proof}
By Lemma \ref{lem:unitary-uhf}, there exists a unital homomorphism $\varphi$ from $C(\mathbb{T})$ to $M_{2^{\infty}}$ such that 
$\tau^{\prime}(\varphi (f)) = \tau_{\omega} (f(u))$ for any $f\in C(\mathbb{T})$, where $\tau^{\prime}$ is the unique tracial state on $M_{2^{\infty}}$. 
For any $n\in\mathbb{N}$, let 
$$ 
v_n := \overbrace{1\otimes \cdots \otimes 1}^n \otimes \varphi (\iota) \otimes 1\otimes \cdots \in \bigotimes_{n\in\mathbb{N}} M_{2^{\infty}}\cong M_{2^{\infty}}. 
$$
Then $(v_n)_n$ is a unitary element in $(M_{2^{\infty}})_{\omega}$ and we have 
$$
\tau_{\omega} (f(u)) = \tau_{\omega} (f([(h_n\otimes v_n)_n]))
$$
for any $f\in C(\mathbb{T})$. Therefore we obtain the conclusion by Theorem \ref{thm:unitary-equivalence}.  
\end{proof}

For a Lipschitz continuous map $U$, we denote by $\mathrm{Lip} (U)$ its Lipschitz constant. 

\begin{lem}\label{lem:full-spectrum}
Let $u$  be a unitary element in $F(\mathcal{W})$ such that $\tau_{\omega} (f(u)) >0$ for any $f\in C(\mathbb{T})_{+}\setminus \{0\}$, and let $z_0\in\mathbb{T}$. Then there exists 
a continuous path of unitaries $U: [0,1] \to F(\mathcal{W})$ such that 
$$
U(0)=z_0 1,\quad U(1)=u, \quad  \mathrm{Lip} (U) \leq \pi .
$$
\end{lem}
\begin{proof}
Let $\{h_n\}_{n\in\mathbb{N}}$ be an approximate unit for $A$. 
By Lemma \ref{lem:induce-uhf-unitary},  there exist a unitary element $(v_n)_n$ in $(M_{2^{\infty}})_{\omega}$ and a unitary element $w$ in 
$F(\mathcal{W})$ such that $wuw^* = [(h_n\otimes v_n)_n]$. 
There exists a continuous path of unitaries $V:[0,1]\to (M_{2^{\infty}})_{\omega}$ such that 
$$
V(0)=z_0 1, \quad V(1)= (v_n)_n, \quad \mathrm{Lip} (U) \leq \pi .
$$
(See, for example, \cite[Lemma 1]{HO}.) 
For any $t\in [0,1]$, let $(v_n(t))_n$ be a representative of $V(t)$. 
Define a continuous path of unitaries $U: [0,1] \to F(\mathcal{W})$ by 
$$
U(t):= w^*[(h_n\otimes v_n(t))_n]w
$$
for any $t\in [0,1]$. Then $U$ has the desired property. 
\end{proof}

The following theorem is the main theorem in this section. 

\begin{thm}\label{thm:homotopy}
Let $u$  be a unitary element in $F(\mathcal{W})$. There exists 
a continuous path of unitaries $U: [0,1] \to F(\mathcal{W})$ such that 
$$
U(0)=1,\quad U(1)=u, \quad \mathrm{Lip} (U) \leq 2\pi .
$$
\end{thm}
\begin{proof}
Let $\varepsilon>0$ and $\delta >0$. 
We denote by $\mu$ the probability measure on $\mathbb{T}$ corresponding to $f \mapsto \tau_{\omega} (f(u))$. 
Since $\mu(\mathbb{T})=1$, there exists an element $z_0$ in $\mathbb{T}$ such that $\mu (\{z\in\mathbb{T}\; |\; |z-z_0|< \delta \})>0$. 
Let $h$ be a positive element in $C(\mathbb{T})$ such that 
$$
\{z\in\mathbb{T}\; |\; |z-z_0|< \delta \} \subset \overline{\mathrm{supp}\; h} \subset \{z\in\mathbb{T}\; |\; |z-z_0|< 2\delta \}. 
$$
Then we have $d_{\tau_{\omega}}(h(u))>0$. Proposition \ref{pro:key-pro} implies that there exists projection $p$ in $\overline{h(u)F(\mathcal{W})h(u)}$ such that 
$\tau_{\omega} (p) >0$. 
Similar arguments as in the proof of \cite[Lemma 1.7]{Cu} show  
that there exist a unitary element $u^{\prime}$ in $(1-p)F(\mathcal{W})(1-p)$ and a continuous path of unitaries $V_1:[0,1]\to F(\mathcal{W})$ such that 
$$
V_1(0) = u^{\prime} + z_0 p, \quad V_1(1)= u, \quad \mathrm{Lip} (V_1) < \varepsilon .
$$
Indeed, we have 
$$
\| u- ((1-p)u(1-p)+ z_0 p) \| < 6\delta 
$$
and, taking a sufficiently small $\delta>0$ and using polar decomposition, we obtain a continuous path of unitaries $V_1$ as above. 
By Lemma \ref{lem:unitary-uhf} and the proof of Lemma \ref{lem:induce-uhf-unitary}, it is easy to see that there exist a unitary element $v$ in $F(\mathcal{W})$ such that 
$\tau_{\omega} (f(v))>0$ for any $f\in C(\mathbb{T})_{+}\setminus \{0\}$.
Using Lemma \ref{lem:full-spectrum}, Lemma \ref{lem:split} and the slow reindexation trick, we may assume that 
$vp=pv$ and $\tau_{\omega}(pv)=\tau_{\omega}(p)\tau_{\omega}(v)$, and we see that there exists a continuous path of unitaries 
$V_2:[0,1]\to F(\mathcal{W})$ such that 
$$
V_2(0) = u^{\prime}+ vp, \quad V_2(1)= u^{\prime}+ z_0 p, \quad \mathrm{Lip} (V_2) \leq \pi .
$$
(See, for example, \cite{Oc} for the slow reindexation trick.) 
Since we have $\tau_{\omega} (f(u^{\prime}+vp))=\tau_{\omega} (f(u^{\prime})) + \tau_{\omega}(f(v))\tau_{\omega}(p)>0$ for  any $f\in C(\mathbb{T})_{+}\setminus \{0\}$, 
it follows from Lemma \ref{lem:full-spectrum} that there exists a continuous path of unitaries 
$V_3:[0,1]\to F(\mathcal{W})$ such that 
$$
V_3(0) = 1, \quad V_3(1)= u^{\prime}+ v p, \quad \mathrm{Lip} (V_3) \leq \pi .
$$
Connecting $V_1$, $V_2$ and $V_3$, we obtain a continuous path of unitaries $U: [0,1] \to F(\mathcal{W})$ such that 
$$
U(0)=1,\quad U(1)=u, \quad \mathrm{Lip} (U) < 2\pi +\varepsilon .
$$
We obtain the conclusion by the usual diagonal argument. 
\end{proof}

We shall show another application of Theorem \ref{thm:unitary-equivalence-ed}. 

\begin{thm}\label{thm:unitary-equivalence-projections}
Let $p$ and $q$ be projections in $F(\mathcal{W})$ such that $0< \tau_{\omega} (p) <1$.  
Then $p$ and $q$ are unitarily equivalent if and only if 
$
\tau_{\omega} (p)= \tau_{\omega} (q)
$. 
\end{thm}

\begin{proof}
The only if part is obvious. We will show the if part. 
Note that the C$^*$-algebra generated by $1$ and $p$ is isomorphic to $C(\{0,1\})\cong \mathbb{C}^2$. 
Since $\tau_{\omega}(p)>0$ and $\tau_{\omega}(1-p)>0$, we have $\tau_{\omega}(f(p))>0$ for any $f\in C(\{0,1\})_{+}\setminus \{0\}$. Also, we have 
$\tau_{\omega}(f(p))=\tau_{\omega}(f(q))$ for any $f\in C(\{0,1\})$. 
There exist positive contractions $(p_n)_n$ and $(q_n)_n$ in $\mathcal{W}_{\omega}$ such that $p=[(p_n)_n]$ and $q=[(q_n)_n]$. 
For any $n\in\mathbb{N}$, define unital c.p. maps $\varphi_n$ and $\psi_n$ from $C(\{0,1\})\cong \mathbb{C}^2$ to $\mathcal{W}^{\sim}$ by 
$\varphi_n((\lambda_1,\lambda_2)):= \lambda_1p_n+ \lambda_2(1-p_n)$ and $\psi_n((\lambda_1,\lambda_2)):= \lambda_1q_n+ \lambda_2(1-q_n)$, respectively. 
Then we have 
$$
| \tau (\varphi_n (f)) - \tau_{\omega} (f(p)) | \to 0, \quad \| [\varphi_n(f),a]\| \to 0, \quad \| [\psi_n(f),a]\| \to 0,   
$$
$$
\| (\varphi_n(fg)- \varphi_n(f)\varphi_n(g))a\| \to 0, \quad \| (\psi_n(fg)- \psi_n (f)\psi_n(g))a\| \to 0,
$$
$$
|\tau (\varphi_n(f))-\tau (\psi_n(f))| \to 0
$$
as $n\to\omega$ for any $f,g\in C(\{0,1\})$, $a\in \mathcal{W}$. 
Therefore the rest of proof is same as the proof of Theorem \ref{thm:unitary-equivalence}. 
\end{proof}

The following corollary is an answer to \cite[Question 2.14]{Kir2}. 

\begin{cor}\label{cor:Kirchberg-question}
The unit $1$ in $F(\mathcal{W})$ is infinite. 
\end{cor}
\begin{proof}
By Proposition \ref{pro:key-pro}, there exist mutually orthogonal non-zero projections $p$ and $q$ in $F(\mathcal{W})$ such that 
$\tau_{\omega}(p)=1/2$ and $\tau_{\omega}(q)=0$. Theorem \ref{thm:unitary-equivalence-projections} implies that $1-p$ is unitarily equivalent 
to $1-(p+q)$. Therefore $1=(1-p)+p$ is Murray-von Neumann equivalent to $(1-(p+q))+p=1-q$.  
Hence $1$ is Murray-von Neumann equivalent to a proper subprojection. Consequently, $1$ is infinite. 
\end{proof}

The corollary above suggests the following question. 

\begin{que}
Let $A$ be a simple separable ($\mathcal{Z}$-stable) stably projectionless C$^*$-algebra. Is $1$ infinite in $F(A)$?
\end{que}

Yuhei Suzuki suggested the following corollary. 

\begin{cor}
The tracial state $\tau_{\omega}$ on $F(\mathcal{W})$ induces an order isomorphism $\tau_{\omega *}$ from $K_0(F(\mathcal{W}))$ onto $\mathbb{R}$. 
\end{cor}
\begin{proof}
Let $p$ and $q$ be projections in $M_\infty (F(\mathcal{W}))$ such that $\tau_{\omega *}([p]-[q])=0$. Using Proposition \ref{pro:key-pro} and Proposition \ref{pro:w-comparison}, 
we see that there exist a natural number $k$, projections $p^{\prime}$, $q^{\prime}$ and $r$ in $F(\mathcal{W})$ such that $[p]=[p^{\prime}]+k[1]+[r]$, $[q]=[q^{\prime}]+k[1]+[r]$ and 
$\tau_{\omega}(p^{\prime})=\tau_{\omega}(q^{\prime})<1$. Moreover, we may assume that there exists a projection $e$ in $F(\mathcal{W})$ such that $p^{\prime},q^{\prime}\in eF(\mathcal{W})e$ 
and $\tau_{\omega}(p^{\prime}) <\tau_{\omega}(e)<1$. Theorem \ref{thm:unitary-equivalence-projections} implies that $p^{\prime}+1-e$ is unitarily equivalent to $q^{\prime}+1-e$ because we have 
$0<\tau_{\omega}(p^{\prime}+1-e)<1$. Then $[p^{\prime}]=[q^{\prime}]$, and hence $[p]=[q]$. Therefore $\tau_{\omega *}$ is injective. 
It follows from Proposition \ref{pro:key-pro} that $\tau_{\omega *}$ is surjective. 
Consequently, $\tau_{\omega *}$ is an order isomorphism from $K_0(F(\mathcal{W}))$ onto $\mathbb{R}$. 
\end{proof}

\section{Rohlin type theorem}\label{sec:Rohlin}

In this section we shall show that every trace scaling automorphism of $\mathcal{W}\otimes\mathbb{K}$ has the Rohlin property. 

\begin{Def}
Let $A$ be a separable C$^*$-algebra, and let $\alpha$ be an automorphism of $A$. We say that $\alpha$ has the 
\textit{Rohlin property} if for any $k\in\mathbb{N}$, there exist 
projections $\{ p_{1,i}\}_{i=0}^{k-1}$ and $\{p_{2,j}\}_{j=0}^{k}$ in $F(A)$ such that 
$$
\sum_{i=0}^{k-1}p_{1,i}+\sum_{j=0}^{k}p_{2,j}=1, \quad \alpha (p_{1,i})=p_{1,i+1}, \quad \alpha (p_{2,j})=p_{2,j+1}
$$
for any $i=0,1,...,k-2$ and $j=0,1,...,k-1$. 
\end{Def}

If $A$ is unital, then the definition above coincides with the usual definition (see, for example, \cite{Kis1}). 

We identify $F(\mathcal{W}\otimes\mathbb{K})$ with $F(\mathcal{W})$. 
We denote by the same symbol $\tau_{\omega}$ the unique tracial state on $F(\mathcal{W}\otimes\mathbb{K})$ for simplicity.  
Note that for any $[(x_n)_n]\in F(\mathcal{W}\otimes\mathbb{K})$, 
$\tau_{\omega} ([(x_n)_n]) = \lim_{n\to\omega} \tau\otimes\mathrm{Tr}  (x_nh) /\tau \otimes \mathrm{Tr} (h)$ 
for some $h\in \mathrm{Ped}\; (\mathcal{W}\otimes\mathbb{K})_{+}\setminus \{0\}$ (see Remark \ref{rem:trace-ideal}). 
The following lemma is a variant of \cite[Theorem 3.4]{MS1}. 

\begin{lem}\label{lem:weak^rohlin}
Let $\alpha$ be a trace scaling automorphism of $\mathcal{W}\otimes\mathbb{K}$. Then for any $k\in\mathbb{N}$, there exists 
a positive contraction $f$ in $F(\mathcal{W}\otimes\mathbb{K})$ such that 
$$
\tau_{\omega} (f)= \frac{1}{k}, \quad f\alpha^{j}(f)=0 
$$
for any $j=1,2,...,k-1$. 
\end{lem}
\begin{proof}
Note that $\pi_{\tau\otimes\mathrm{Tr}}(\mathcal{W}\otimes\mathbb{K})^{''}$ is the AFD factor of type II$_{\infty}$ and 
$\tilde{\alpha}$ is a trace scaling automorphism. 
Hence it follows from \cite[Lemma 5]{C2} and \cite[Theorem 1.2.5]{C2} that there exist projections $\{\tilde{p}_{j}\}_{j=1}^{k}$ in 
$(\pi_{\tau\otimes\mathrm{Tr}}(\mathcal{W}\otimes\mathbb{K})^{''})_{\omega}$ 
such that 
$$
\sum_{j=1}^k \tilde{p}_j=1, \quad \tilde{\alpha} (\tilde{p}_j) = \tilde{p}_{j+1}
$$
for any $j=1,2,...,k-1$. 
Proposition \ref{pro:sato} implies that there exists a positive contraction $e$ in $F(\mathcal{W}\otimes\mathbb{K})$ such that 
$\varrho_{\mathcal{W}\otimes\mathbb{K}} (e)= \tilde{p}_1$. It is easy to see that $\tau_{\omega} (e)=1/k$. 
Let $(e_n)_n$ be a representative of $e$.  
Then 
$$
\| \pi_{\tau\otimes\mathrm{Tr}}(e_n\alpha^{j}(e_n) )\hat{h} \|_2  \to 0 
$$
as $n\to \omega$ for any $h\in \mathrm{Ped}(\mathcal{W}\otimes\mathbb{K})$ and $j=1,2,...,k-1$. 
By similar arguments as in the proof of \cite[Proposition 3.3]{MS1}, one can prove the lemma. 
Indeed, put 
$$
e^{\prime}_n := e_n^{\frac{1}{2}}\left( \sum_{j=1}^{k-1}\alpha^{j}(e_n)\right)e_n^{\frac{1}{2}}.
$$
Then $\|\pi_{\tau\otimes\mathrm{Tr}}(e_n^{\prime})\hat{h}\|_2 \to 0$ as $n\to \omega$ for any $h\in \mathrm{Ped}(\mathcal{W}\otimes\mathbb{K})$. 
For any $\varepsilon >0$, define 
$$
g_{\varepsilon}(t):= \left\{\begin{array}{cl}
\varepsilon^{-1}t & \text{if } t\in [0,\varepsilon]   \\
1                 & \text{if } t\in [\varepsilon, \infty ) 
\end{array}
\right.
$$
and let $f_n:= e_n- e_n^{\frac{1}{2}}g_{\varepsilon}(e_n^{\prime})e_n^{\frac{1}{2}}$. 
The same proof as \cite[Proposition 3.3]{MS1} shows that 
$$
\| f_n\alpha^{j} (f_n) \|^2 < \varepsilon 
$$
for any $j=1,2,...,k-1$. 
For $h\in \mathrm{Ped}(\mathcal{W}\otimes\mathbb{K})_{+,1}$, we have 
\begin{align*}
\lim_{n\to\omega} \|\pi_{\tau\otimes\mathrm{Tr}} (e_n-f_n)\hat{h}\|_2 
& = \lim_{n\to\omega} \|\pi_{\tau\otimes\mathrm{Tr}}( e_n^{\frac{1}{2}}g_{\varepsilon}(e_n^{\prime})e_n^{\frac{1}{2}})\hat{h}\|_2 \\ 
& \leq \lim_{n\to\omega} \|\pi_{\tau\otimes\mathrm{Tr}}(g_{\varepsilon}(e_n^{\prime})e_n^{\frac{1}{2}})\hat{h}\|_2 \\
& = \lim_{n\to\omega} \|\pi_{\tau\otimes\mathrm{Tr}}(g_{\varepsilon}(e_n^{\prime})h^{\frac{1}{2}}e_n^{\frac{1}{2}})\widehat{h^{\frac{1}{2}}}\|_2 \\
& \leq  \lim_{n\to\omega} \|\pi_{\tau\otimes\mathrm{Tr}}(g_{\varepsilon}(e_n^{\prime}))\widehat{h^{\frac{1}{2}}}\|_2=0.
\end{align*}
Hence $\tau_{\omega} ([(f_n)_n])=\tau_{\omega} ([(e_n)_n])=1/k$ (see Remark \ref{rem:trace-ideal}). Therefore we obtain the conclusion by the usual diagonal argument. 
\end{proof}

\begin{lem}\label{lem:cyclic}
Let $\alpha$ be a trace scaling automorphism of $\mathcal{W}\otimes\mathbb{K}$. Then for any $k\in\mathbb{N}$, there exists 
a projection $p$ in $F(\mathcal{W}\otimes\mathbb{K})$ such that  
$$
\tau_{\omega} (p)= \frac{1}{k}, \quad p\alpha^{j}(p)=0 
$$
for any $j=1,2,...,k-1$, and $p$ is Murray-von Neumann equivalent to $\alpha(p)$
\end{lem}
\begin{proof} 
By Lemma \ref{lem:weak^rohlin}, there exists 
a positive contraction $f$ in $F(\mathcal{W}\otimes\mathbb{K})$ such that 
$$
\tau_{\omega} (f)= \frac{1}{k}, \quad f\alpha^{j}(f)=0 
$$
for any $j=1,2,...,k-1$. Since $f$ is a contraction, $d_{\tau_{\omega}} (f) \geq \tau_{\omega}(f)=1/k$.
Let $\varepsilon >0$. (Note that we may assume $0<1/k-\varepsilon <1$.) By Proposition \ref{pro:key-pro}, there exists a projection $p$ in $\overline{fF(\mathcal{W}\otimes\mathbb{K})f}$ 
such that $\tau_{\omega}(p)=1/k-\varepsilon$. It is easy to see that $p\alpha^j(p)=0$ for any $j=1,2,....k-1$. Since the tracial state on $F(\mathcal{W}\otimes\mathbb{K})\cong F(\mathcal{W})$ 
is unique by Proposition \ref{pro:w-comparison}, $\tau_{\omega}\circ \alpha= \tau_{\omega}$, and hence $\tau_{\omega}(p)=\tau_{\omega}(\alpha(p))$. 
Therefore Theorem \ref{thm:unitary-equivalence-projections} implies that $p$ is Murray-von Neumann equivalent to $\alpha(p)$. 
Consequently, we obtain the conclusion by the usual diagonal argument. 
\end{proof}

The following theorem is the main theorem in this section. The proof is based on \cite{Kis0} and \cite{Kis00}.

\begin{thm}\label{thm:rohlin-type-scale}
Let $\alpha$ be a trace scaling automorphism of $\mathcal{W}\otimes\mathbb{K}$. Then $\alpha$ has the Rohlin property. 
\end{thm}
\begin{proof}
For any $N\geq 2$, it follows from Lemma \ref{lem:cyclic} that there exists a (non-unital) homomorphism $\varphi$ from $M_{N}(\mathbb{C})$ to 
$F(\mathcal{W}\otimes\mathbb{K})$ such that 
$$
\alpha (\varphi (e_{ij})) = \varphi (e_{i+1j+1})
$$
for any $i,j=1,...,N-1$, where $\{e_{ij}\}_{i,j=1}^{N}$ is the standard matrix units of $M_{N}(\mathbb{C})$ (see \cite[Lemma 4.3]{Kis0} for details). 
By the same argument as in the proof of \cite[Lemma 2.1]{Kis0} and the usual diagonal argument, we see that for any $m\in\mathbb{N}$, 
there exists a projection $p$ in $F(\mathcal{W}\otimes\mathbb{K})$ such that 
$$
\tau_{\omega} (p)=\frac{1}{m}, \quad p\alpha^{j}(p)=0, \quad \alpha^{m}(p)=p 
$$
for any $j=1,2,...,m-1$. Note that there exists an element $v$ in $F(\mathcal{W}\otimes\mathbb{K})\cong F(\mathcal{W})$ such that 
$$
v^*v= 1-\sum_{j=0}^{m-1}\alpha^{j}(p), \quad vv^* \leq p 
$$
because $\mathcal{W}$ has property (SI) and we have $\tau_{\omega} (1-\sum_{j=0}^{m-1}\alpha^{j}(p))=0$. 
Therefore the rest of the proof is the same as \cite[Theorem 2.1]{Kis0} and \cite[Lemma 4.4]{Kis00}.  
\end{proof}

An automorphism $\alpha$ of $\mathcal{W}$ is said to be \textit{strongly outer} if $\tilde{\alpha}$ is not inner in $\pi_{\tau} (\mathcal{W})^{''}$. 
The same proof as Lemma \ref{lem:weak^rohlin} shows the following lemma.

\begin{lem}
Let $\alpha$ be an automorphism of $\mathcal{W}$ such that $\alpha^{m}$ is strongly outer for any $m\in\mathbb{Z}\setminus\{0\}$. Then for any $k\in\mathbb{N}$, there exists 
a positive contraction $f$ in $F(\mathcal{W})$ such that 
$$
\tau_{\omega} (f)= \frac{1}{k}, \quad f\alpha^{j}(f)=0 
$$
for any $j=1,2,...,k-1$. 
\end{lem}

\begin{lem}
Let $\alpha$ be an automorphism of $\mathcal{W}$ such that $\alpha^{m}$ is strongly outer for any  $m\in\mathbb{Z}\setminus\{0\}$. Then for any $k\in\mathbb{N}$, there exists 
a projection $p$ in $F(\mathcal{W})$ such that such that 
$$
\tau_{\omega} (p)= \frac{1}{k}, \quad p\alpha^{j}(p)=0 
$$
for any $j=1,2,...,k-1$, and $p$ is Murray-von Neumann equivalent to $\alpha(p)$. 
\end{lem}
\begin{proof}
In a similar way as in Lemma \ref{lem:cyclic}, we see that there exists 
a projection $p$ in $F(\mathcal{W})$ such that  
$$
\tau_{\omega} (p)= \frac{1}{k}, \quad p\alpha^{j}(p)=0 
$$
for any $j=1,2,...,k-1$. 
Since every automorphism of $\mathcal{W}$ is approximately inner by Theorem \ref{thm:Razak}, we obtain the conclusion by \cite[Lemma 4.3]{M1}. 
\end{proof}

By the same arguments as in the proof of Theorem \ref{thm:rohlin-type-scale}, we obtain the following theorem. 

\begin{thm}\label{thm:rohlin-type}
Let $\alpha$ be an automorphism of $\mathcal{W}$ such that $\alpha^{m}$ is strongly outer for any  $m\in\mathbb{Z}\setminus\{0\}$. Then $\alpha$ has the Rohlin property. 
\end{thm}

\section{Outer conjugacy}\label{sec:Outer}

In this section we shall classify trace scaling automorphisms of $\mathcal{W}\otimes\mathbb{K}$ up to outer conjugacy. 
Using Theorem \ref{thm:homotopy} instead of \cite[Lemma 1]{HO}, we can prove the following theorem by essentially the same argument as in the proof of 
\cite[Theorem 1]{HO}. (See also \cite{I0} and \cite{Kis3}.) 

\begin{thm}
Let $A$ be a C$^*$-algebra which is isomorphic to $\mathcal{W}$ or $\mathcal{W}\otimes\mathbb{K}$, 
and let $\alpha$ be an automorphism of $A$ with the Rohlin property.  For any unitary element $u$ in $F(A)$, 
there exists a unitary element $v$ in $F(A)$ such that $u=v\alpha(v)^*$. 
\end{thm}

The following lemma is an immediate consequence of the theorem above and Corollary \ref{cor:u-lift}. 

\begin{lem}\label{lem:stability}
Let $A$ be a C$^*$-algebra which is isomorphic to $\mathcal{W}$ or $\mathcal{W}\otimes\mathbb{K}$, 
and let $\alpha$ be an automorphism of $A$ with the Rohlin property. 
Then 
for any finite subsets $E\subset A^{\sim}$, $F\subset A$ and $\varepsilon >0$, there exist a finite subset $G\subset A$ and 
$\delta >0$ such that the following holds. If $u$ is a unitary element in $M(A)$ satisfying 
$$
\| [u,a] \| < \delta 
$$
for any $a\in G$, then there exists a unitary element $v$ in $A^{\sim}$ such that 
$$
\| [v,x] \| < \varepsilon, \quad \| (u -v\alpha (v)^* )y \| < \varepsilon, \quad  \| y(u -v\alpha (v)^* ) \| < \varepsilon
$$
for any $x\in E$ and $y\in F$. 
\end{lem}

The following theorem is the main theorem in this paper. 

\begin{thm}\label{thm:main-i}
Let $\alpha$ and $\beta$ be trace scaling automorphisms of $\mathcal{W}\otimes\mathbb{K}$. Then $\alpha$ and $\beta$ are outer conjugate if and only if 
$\lambda (\alpha)=\lambda (\beta)$. 
\end{thm}
\begin{proof}
The only if part is obvious. We will show the if part. 
Theorem \ref{thm:rohlin-type-scale} implies that $\alpha$ and $\beta$ have the Rohlin property. 
Since $\lambda (\alpha)= \lambda (\beta)$, $\alpha$ is approximately unitarily equivalent to $\beta$ by Theorem \ref{thm:Razak}. 
Therefore we obtain the conclusion by Lemma \ref{lem:stability} and the Bratteli-Elliott-Evans-Kishimoto intertwining argument \cite{EK} 
(see also \cite{I1}, \cite{Kis1}, \cite{M}, \cite{Nak} and \cite{Sa} for similar arguments). 
Indeed, let $\{x_n\}_{n\in\mathbb{N}}$ be a dense set in the unit ball of $\mathcal{W}\otimes\mathbb{K}$. By induction, we shall construct 
sequences of automorphisms $\{\alpha_{2n}\}_{n=0}^\infty$, $\{\beta_{2n+1}\}_{n=0}^\infty$ of $\mathcal{W}\otimes\mathbb{K}$ and sequences of 
unitaries $\{u_n\}_{n=0}^\infty$, $\{v_n\}_{n=0}^\infty$, $\{w_n\}_{n=0}^\infty$, $\{\widetilde{w}_n\}_{n=0}^\infty$ in $(\mathcal{W}\otimes\mathbb{K})^{\sim}$ as follows: 
Put $\alpha_{0}:=\alpha$, $\beta_{1}:=\beta$, and let $F_1:=\{x_1,x_1^* \}$, $F_1^{\prime}:=\{x_1,x_1^* \}$, $E_1:=\{1 \}$. 
Applying Lemma \ref{lem:stability} to $\beta_1$, $E_1$, $F_1^{\prime}$ and $1/2$, we obtain a finite subset $G_1\subset \mathcal{W}\otimes\mathbb{K}$ 
and $\delta_1>0$. Set 
$$
F_2:= \beta_1^{-1}(G_1)\cup F_1 \cup \{x_2, x_2^* \}.
$$
Since $\beta_1$ is approximately unitarily equivalent to $\alpha_0$, there exists a unitary element $u_0$ in $(\mathcal{W}\otimes\mathbb{K})^{\sim}$ such that 
$$
\| \beta_1 (a) -u_0 \alpha_0(a) u_0^* \| < \frac{\delta_1}{2} \eqno{(1)} 
$$
for any $a\in F_2$. Put $\alpha_{2}:= \mathrm{Ad} (u_0) \circ \alpha_{0}$, $v_0:=u_0$, and let $w_0:=u_0\alpha_{0}(v_0)v_0^*$, $\widetilde{w}_0:=w_0$. 
Set
$$
E_{2} := F_2\cup \mathrm{Ad} (v_0)(F_2)\cup \{\widetilde{w}_0\}, \quad F_2^{\prime}:= F_2\cup F_2\widetilde{w}_0^*. 
$$
Applying Lemma \ref{lem:stability} to $\alpha_2$, $E_2$, $F_2^{\prime}$ and $1/2^{2}$, we obtain a finite subset $G_2\subset \mathcal{W}\otimes\mathbb{K}$ 
and $\delta_2>0$. We may assume that $\delta_2 < \delta_1/2$. Set 
$$
F_3:= \alpha_2^{-1}(G_2)\cup F_2 \cup \{x_3, x_3^* \}.
$$
Since $\alpha_2$ is approximately unitarily equivalent to $\beta_1$, there exists a unitary element $u_1$ in $(\mathcal{W}\otimes\mathbb{K})^{\sim}$ such that 
$$
\| \alpha_2 (a) -u_1 \beta_1(a) u_1^* \| < \frac{\delta_2}{2} \eqno{(2)} 
$$
for any $a\in F_2$. Put $\beta_{3}:= \mathrm{Ad} (u_1)\circ \beta_1$. By $(1)$ and $(2)$, we have 
$$
\| [u_1,a] \| < \delta_1
$$
for any $a\in G_1$. Hence there exists a unitary element $v_1$ such that 
$$
\| [v_1,x] \| < \frac{1}{2}, \quad \| y(u_1- v_1\beta_1(v_1)^*) \| < \frac{1}{2}
$$
for any $x\in E_1$ and $y\in F_1^{\prime}$ by Lemma \ref{lem:stability}. Put $w_1:=u_1\beta_{1}(v_1)v_1^*$, $\widetilde{w}_1:=w_{1}$, and set 
$$
E_{3} := F_3\cup \mathrm{Ad} (v_1)(F_3)\cup \{\widetilde{w}_1\}, \quad F_3^{\prime}:= F_3\cup F_3\widetilde{w}_1^*.
$$
Applying Lemma \ref{lem:stability} to $\beta_3$, $E_3$, $F_3^{\prime}$ and $1/2^{3}$, we obtain a finite subset $G_3\subset \mathcal{W}\otimes\mathbb{K}$ 
and $\delta_3>0$. We may assume that $\delta_3 < \delta_2/2$. Set 
$$
F_4:= \beta_3^{-1}(G_3)\cup F_3 \cup \{x_4, x_4^* \}.
$$
Since $\beta_3$ is approximately unitarily equivalent to $\alpha_2$, there exists a unitary element $u_2$ in $(\mathcal{W}\otimes\mathbb{K})^{\sim}$ such that 
$$
\| \beta_3 (a) -u_2 \alpha_2(a) u_2^* \| < \frac{\delta_3}{2} \eqno{(3)} 
$$
for any $a\in F_3$. Put $\alpha_{4}:= \mathrm{Ad} (u_2)\circ \alpha_2$. By $(2)$ and $(3)$, we have 
$$
\| [u_2,a] \| < \delta_2
$$
for any $a\in G_2$. Hence there exists a unitary element $v_2$ such that 
$$
\| [v_2,x] \| < \frac{1}{2^2}, \quad \| y(u_2- v_2\alpha_2(v_2)^*) \| < \frac{1}{2^2}
$$
for any $x\in E_2$ and $y\in F_2^{\prime}$ by Lemma \ref{lem:stability}. 
Put $w_2:=u_2\alpha_{2}(v_2)v_2^*$, and let $\widetilde{w}_2:= w_2v_2\widetilde{w}_{0}v_2^*$. 
Set 
$$
E_{4} := F_4\cup \mathrm{Ad} (v_2 v_0)(F_4)\cup \{\widetilde{w}_2\}, \quad F_4^{\prime}:= F_4\cup F_4\widetilde{w}_2^*.
$$
Applying Lemma \ref{lem:stability} to $\alpha_4$, $E_4$, $F_4^{\prime}$ and $1/2^{4}$, we obtain a finite subset $G_4\subset \mathcal{W}\otimes\mathbb{K}$ 
and $\delta_4>0$. We may assume that $\delta_4 < \delta_3/2$. Set 
$$
F_5:= \alpha_4^{-1}(G_4)\cup F_4 \cup \{x_5, x_5^* \}.
$$
Since $\alpha_4$ is approximately unitarily equivalent to $\beta_3$, there exists a unitary element $u_3$ in $(\mathcal{W}\otimes\mathbb{K})^{\sim}$ such that 
$$
\| \alpha_4 (a) -u_3 \beta_3(a) u_3^* \| < \frac{\delta_4}{2} \eqno{(4)} 
$$
for any $a\in F_4$. Put $\beta_{5}:= \mathrm{Ad} (u_3)\circ \beta_3$. By $(3)$ and $(4)$, we have 
$$
\| [u_3,a] \| < \delta_3
$$
for any $a\in G_3$. Hence there exists a unitary element $v_3$ such that 
$$
\| [v_3,x] \| < \frac{1}{2^3}, \quad \| y(u_3- v_3\beta_3(v_3)^*) \| < \frac{1}{2^3}
$$
for any $x\in E_3$ and $y\in F_3^{\prime}$ by Lemma \ref{lem:stability}. 
Put $w_3:=u_3\beta_{3}(v_3)v_3^*$, and let $\widetilde{w}_3:= w_3v_3\widetilde{w}_{1}v_3^*$. 

Repeating this process, we obtain sequences $\{\alpha_{2n}\}_{n=0}^\infty$, $\{\beta_{2n+1}\}_{n=0}^\infty$, $\{u_n\}_{n=0}^\infty$, 
$\{v_n\}_{n=0}^\infty$, $\{w_n\}_{n=0}^\infty$ and $\{\widetilde{w}_n\}_{n=0}^\infty$ such that 
\begin{align*}
& (\mathrm{i})\; \alpha_{2n}= \mathrm{Ad}(u_{2n-2})\circ \alpha_{2n-2}, \quad \beta_{2n+1}= \mathrm{Ad}(u_{2n-1})\circ \beta_{2n-1}, \\
& (\mathrm{ii})\; w_{2n}= u_{2n}\alpha_{2n}(v_{2n})v_{2n}^*, \quad w_{2n+1}= u_{2n+1}\beta_{2n+1}(v_{2n+1})v_{2n+1}^*, \\ 
&  \quad\; \widetilde{w}_{n+1}=w_{n+1}v_{n+1}\widetilde{w}_{n-1}v_{n+1}^*, \\
& (\mathrm{iii})\; \|\beta_{2n-1} (x_i)- \alpha_{2n}(x_i) \| <\frac{\delta_{2n-1}}{2}, \quad 1\leq \forall i \leq 2n, \\
& (\mathrm{iv})\; \|\alpha_{2n} (x_i)- \beta_{2n+1}(x_i) \| <\frac{\delta_{2n}}{2}, \quad 1\leq \forall i \leq 2n+1, \\
& (\mathrm{v})\; \| [v_{2n}, x_i] \| < \frac{1}{2^{2n}}, \quad \| [v_{2n}, \mathrm{Ad} (v_{2n-2}v_{2n-4}\cdots v_{0})(x_i)] \| < \frac{1}{2^{2n}}, \\
&  \quad\;  \|[v_{2n}, \widetilde{w}_{2n-2}] \| < \frac{1}{2^{2n}},   \quad 1\leq \forall i \leq 2n, \\
& (\mathrm{vi})\; \| [v_{2n+1}, x_i] \| < \frac{1}{2^{2n+1}}, \quad \| [v_{2n+1}, \mathrm{Ad} (v_{2n-1}v_{2n-3}\cdots v_{1})(x_i)] \| < \frac{1}{2^{2n+1}}, \\
&  \quad\; \|[v_{2n+1}, \widetilde{w}_{2n-1}] \| < \frac{1}{2^{2n+1}}, \quad 1\leq \forall i \leq 2n+1, \\
& (\mathrm{vii})\;  \| x_i(u_{2n}- v_{2n}\alpha_{2n}(v_{2n})^*) \| < \frac{1}{2^{2n}},  \\
& \quad\; \| x_i^*\widetilde{w}_{2n-2}^*(u_{2n}- v_{2n}\alpha_{2n}(v_{2n})^*) \| < \frac{1}{2^{2n}}, \quad 1\leq \forall i \leq 2n,  \\
& (\mathrm{viii})\; \| x_i(u_{2n+1}- v_{2n+1}\beta_{2n+1}(v_{2n+1})^*) \| < \frac{1}{2^{2n+1}},  \\
& \quad\; \| x_i^*\widetilde{w}_{2n-1}^*(u_{2n+1}- v_{2n+1}\beta_{2n+1}(v_{2n+1})^*) \| < \frac{1}{2^{2n+1}}, \quad 1\leq \forall i \leq 2n+1,  
\end{align*}
for any $n\in\mathbb{N}$, where $\{\delta_n\}_{n\in\mathbb{N}}$ is a sequence of positive numbers such that $\delta_n< \delta_{n-1}/2$. 

For any $n\in\mathbb{N}$, define $\theta_{n}:=\mathrm{Ad}(v_{2n}v_{2n-2}\cdots v_{0})$ and $\gamma_{n}:=\mathrm{Ad}(v_{2n+1}v_{2n-1}\cdots v_{1})$. 
By (v), (vi) and the same proof as \cite[Theorem 3.5]{I1}, the point-norm limit maps 
$\theta = \lim_{n\to\infty} \theta_{n}$ and $\gamma= \lim_{n\to\infty} \gamma_n$ exist and define automorphisms on $\mathcal{W}\otimes\mathbb{K}$. 
 
For any $n\in\mathbb{N}$, (ii) and (vii) imply that  
$$
\| x_i(w_{2n}-1)\| < \frac{1}{2^{2n}}, \quad \| (w_{2n}-1)\widetilde{w}_{2n-2}x_i \| < \frac{1}{2^{2n}}
$$
for any $i=1,2....,2n$. By (ii) and (v), we have for any $n\in\mathbb{N}$, 
\begin{align*}
\| x_{i}  (\widetilde{w}_{2n}-\widetilde{w}_{2n-2}) \| 
& = \| x_i w_{2n}v_{2n}\widetilde{w}_{2n-2}v_{2n}^*- x_i \widetilde{w}_{2n-2}v_{2n}v_{2n}^*\| \\
& \leq \| x_i w_{2n}[v_{2n},\widetilde{w}_{2n-2}]\| + \| x_{i}(w_{2n}-1)\widetilde{w}_{2n-2}v_{2n} \| \\
& < \frac{1}{2^{2n-1}}
\end{align*}
and
\begin{align*}
\|  (\widetilde{w}_{2n}-\widetilde{w}_{2n-2})x_i \| 
& = \| w_{2n}v_{2n}\widetilde{w}_{2n-2}v_{2n}^*x_i- \widetilde{w}_{2n-2}v_{2n}v_{2n}^*x_i\| \\ 
& \leq \| w_{2n}[v_{2n}, \widetilde{w}_{2n-2}]v_{2n}^*x_i\| +\| (w_{2n}-1)\widetilde{w}_{2n-2}x_{i}\| \\
& < \frac{1}{2^{2n-1}} 
\end{align*}
for any $i=1,2....,2n$. Therefore $\{\widetilde{w}_{2n}\}_{n\in\mathbb{N}}$ is a strict Cauchy sequence of unitaries in $(\mathcal{W}\otimes\mathbb{K})^{\sim}$. 
Since $M(\mathcal{W}\otimes\mathbb{K})$ is strictly complete, there exists a unitary element $w_0^{\prime}$ in $M(\mathcal{W}\otimes\mathbb{K})$ such that 
$\{\widetilde{w}_{2n}\}_{n\in\mathbb{N}}$ converges strictly to $w_0^{\prime}$. In a similar way, we see that there exists a unitary element $w_1^{\prime}$ in 
$M(\mathcal{W}\otimes\mathbb{K})$ such that $\{\widetilde{w}_{2n+1}\}_{n\in\mathbb{N}}$ converges strictly to $w_1^{\prime}$. 

It can be easily checked that 
$$
\alpha_{2n+2}= \mathrm{Ad}(\widetilde{w}_{2n}) \circ \theta_{n}\circ \alpha \circ \theta_{n}^{-1}, \quad \beta_{2n+3}= \mathrm{Ad}(\widetilde{w}_{2n+1}) \circ \gamma_{n}\circ \beta \circ \gamma_{n}^{-1}
$$
for any $n\in\mathbb{N}$. It follows from (iv) that for any $n\in\mathbb{N}$, we have 
$$
\| \alpha_{2n+2}(x_i) - \beta_{2n+3}(x_i) \| < \frac{\delta_{2n+2}}{2}
$$
for any $i=1,2,...,2n+3$. 
Therefore we see that 
$$
\mathrm{Ad}(w_0^{\prime})\circ \theta \circ \alpha \circ \theta^{-1} (x) = \mathrm{Ad}(w_1^{\prime})\circ \gamma \circ \beta \circ \gamma^{-1} (x) 
$$
for any $x\in \mathcal{W}\otimes\mathbb{K}$ because $\{\widetilde{w}_{n}\}_{n\in\mathbb{N}}$ is a bounded sequence and $\lim_{n\to\infty}\delta_n=0$. 
\end{proof}

By Theorem \ref{thm:rohlin-type} and the same proof as above, we obtain the following theorem. 

\begin{thm}\label{thm:main-ii}
Let $\alpha$ and $\beta$ be automorphisms of $\mathcal{W}$. If $\alpha^m$ and $\beta^m$ are strongly outer for any $m\in\mathbb{Z}\setminus\{0\}$, then 
$\alpha$ and $\beta$ are outer conjugate. 
\end{thm}

%
%

\section*{Acknowledgments}
The author would like to thank Hiroki Matui for many helpful discussions and valuable suggestions. 
He is also grateful to the people in University of M\"unster, where a part of this work was done, 
for their hospitality.

\end{document}